\documentclass[10pt,english]{amsart}

\usepackage{amsfonts,latexsym}

\newtheorem{theorem}{Theorem}[section]
\newtheorem{proposition}{Proposition}[section]

\newtheorem{lemma} {Lemma}[section]
\newtheorem{remark}{Remark}[section]
\newtheorem{definition}{Definition}[section]

\newtheorem{example}{Example}[section]

\numberwithin{equation}{section}

\begin{document}

\title{Stable solutions of the Allen-Cahn equation in dimension $8$ and minimal cones}

\author{Frank Pacard}
\address{Frank Pacard. Centre de Math\'ematiques Laurent Schwartz, \'Ecole Polytechnique, 91128 Palaiseau, France.}
\email{frank.pacard@math.polytechnique.fr}

\author{Juncheng Wei}
\address{Juncheng Wei. Department of Mathematics, The Chinese University of Hong Kong, Shatin, Hong Kong.}
\email{wei@math.cuhk.edu.hk}

\thanks{{\bf Acknowledgments :} The first  author is supported by an Earmarked Grant of RGC  of Hong Kong and the second author is partially supported by the ANR-08-BLANC-0335-01 grant.}


\begin{abstract} 
For all $n \geq 1$, we are interested in bounded solutions of the Allen-Cahn equation $\Delta u  + u - u^3 = 0$ which are defined in all $\mathbb R^{n+1}$  and whose zero set is asymptotic to a given minimal cone. In particular, in dimension $n+1 \geq 8$, we prove the  existence of {\em stable solutions} of the Allen-Cahn equation whose zero sets are not  hyperplanes. 
\end{abstract}

\date{\today}\maketitle

\setcounter{equation}{0}

\section{Introduction and statement of main results}

This paper is focussed on the construction of entire solutions of the Allen-Cahn equation
\begin{equation}
	\Delta u + u - u^3 = 0 ,
\label{ac}
\end{equation}
in $\mathbb R^{n+1}$, with $n \geq 1$. This equation is the Euler-Lagrange equation of the energy functional 
\begin{equation} 
	E (u) : = \frac 1 2 \, \int_\Omega|\nabla u|^2 \, {\rm dx}   + \frac 1 4 \, \int_\Omega (1-u^2)^2 \, {\rm dx},
\label{energy}
\end{equation}
which arises in the gradient theory of phase transitions. We refer to \cite{All-Cah} for further motivation and references on the subject.

It is well known that there are strong links and analogies between the study of (\ref{ac}) and the theory of embedded minimal hypersurfaces (see \cite{FV} for a detailed discussion on these analogies). These analogies led de Giorgi \cite{dg} to formulate in 1978 the following celebrated conjecture concerning the classification of entire solutions of (\ref{ac}) which are monotone in one direction.

{\bf The de Giorgi's conjecture~:} {\em  Assume that $u$ is a bounded  solution of  (\ref{ac}) which is defined in $\mathbb R^{n+1}$ and which is monotone in one direction (without loss of generality, we can assume that $\partial_{x_{n+1}}  u > 0$). At least when the dimension is less than or equal to $8$, the level sets of $u$ should be hyperplanes.}

In dimension $1$, equation (\ref{ac}) reduces to an autonomous second order nonlinear differential equation
\begin{equation}
	u'' + u - u^3 = 0, 
\label{eq:fp1}
\end{equation}
whose first integral is given by
\[
4 \, u' \, ^2 - (1- u^2)^2 = c ,
\]
for some $c \in \mathbb R$. It is well known and in fact easy to check that the function
\begin{equation}
	u_1 (t) :=  \tanh \left ( \frac t {\sqrt{2}} \right ) ,
\label{eq:fp2}
\end{equation} 
is the unique solution of (\ref{eq:fp1}) which vanishes at $t=0$ and tends to $1$ (rep.$-1$) at  $+\infty$ (resp. $- \infty$).  According to the {\em de Giorgi's Conjecture}, in dimension less than or equal to $8$, bounded, entire solutions of (\ref{ac}) which are monotone in one direction should be of the form
\begin{equation}
	u(x) = u_1 ( x \cdot a + b ),
\label{formu}
\end{equation}
for some $b\in \mathbb R$ and some vector $a\in \mathbb R^{n+1}$ with $|a| =1 $.  

The {\em de Giorgi's Conjecture} is the natural counterpart of the famous {\em Bernstein problem} for minimal hypersurfaces in Euclidean space, which states that a minimal graph in $\mathbb R^{n+1}$ must  be a hyperplane if  the dimension of the ambient space is smaller than or equal to $8$. In a celebrated paper, Bombieri, de Giorgi and Giusti \cite{bdg} proved that, when the dimension of the ambient space if larger than or equal to $9$, one can indeed find minimal graphs which are not hyperplanes, showing that the statement of Bernstein's problem is sharp. 

Some important advances have been achieved in recent years concerning the resolution of the {\em de Giorgi's Conjecture}. In particular, the conjecture has been fully established in dimension $2$ by Ghoussoub and Gui \cite{gg} and in dimension $3$ by Ambrosio and Cabr\'e \cite{cabre}. In dimensions $4$ and $5$, partial results have been obtained by Ghoussoub and Gui \cite{gg2}. Finally, let us mention that Savin  \cite{savin} has established its validity when $4 \le n+1 \le 8$ under the additional assumption that
\begin{equation}
\label{gibbs}
	\lim_{x_{n+1} \to \pm \infty} u(x', x_{n+1})= \pm 1,
\end{equation}
for all $x' \in \mathbb R^n$. We refer to \cite{aac} where this condition is discussed at length.

In \cite{dkwdg}, del Pino, Kowalczyk and Wei have constructed entire solutions  of (\ref{ac}) in dimension $9$ which are monotone in one direction and whose level sets are not hyperplanes starting from the entire minimal graph found by Bombieri, de Giorgi and Giusti \cite{bdg}. This result shows that the statement of the {\em de Giorgi's Conjecture}, the bound on the dimension is sharp.

The extra condition (\ref{gibbs}), used by Savin, is related to the  so-called~:

{\bf Gibbons' conjecture :}  {\em  Assume that $u$ is a bounded solution of (\ref{ac}) satisfying
\begin{equation}
\label{gibbs1} 
	\lim_{x_{n+1} \to \pm \infty} u(x', x_{n+1})= \pm 1,
\end{equation}
uniformly in $x' \in \mathbb R^n$. Then, the level sets of $u$ should be hyperplanes.}

{\em Gibbons' Conjecture} has been proved in all dimensions  with different methods by Caffarelli and C\'ordoba \cite{caffarellicordoba},  Farina \cite{Fa}, Barlow, Bass and Gui \cite{BBG}, and Berestycki,  Hamel, and Monneau \cite{BHM}.  In \cite{caffarellicordoba} and in \cite{BBG},  it is also proven that the conjecture is true for any  solution that has one level set which is  globally a Lipschitz graph over $\mathbb R^n$. 

In the present paper, we are interested in the understanding of {\em stable solutions} of (\ref{ac}) and hence we start with the~: 
\begin{definition}
We will say that $u$, solution of (\ref{ac}), is {\em stable} if
\begin{equation} 
	\int_{\mathbb R^{n+1}}  ( |\nabla \psi|^2  - \psi^2 + 3\, u^2 \,  \psi^2 ) \, {\rm dx}  \geq 0 ,
\label{QQ}
\end{equation}
for any smooth function $\psi$ with compact support in $\mathbb R^{n+1}$.
\label{de:111}
\end{definition}

In dimensions $2$ and $3$, the stability property turns out to be a key ingredient in the proof of the {\em de Giorgi's  Conjecture} which is given in \cite{cabre}, \cite{gg} and, as observed by Dancer \cite{dancer}, the stability assumption is indeed a sufficient condition to classify solutions of (\ref{ac}) in dimensions $2$ and prove that, at least in these dimensions, the solutions of (\ref{ac}) satisfying (\ref{QQ}) are given by (\ref{formu}).

The monotonicity assumption in the {\em de Giorgi's Conjecture} implies the stability of the solution. Indeed, if $u$ is a solution of (\ref{ac}) such that $ \partial_{x_{n+1}} u > 0$, then $u$ is stable in the above sense. To prove this result, we consider  the linearized operator
\[
	L : = - (\Delta  + 1 - 3 \, u^2 ),
\]
about the solution $u$. If $\phi >0$ is a smooth function which is a solution of 
\[
	L \, \phi = 0, 
\]
then, one can multiply this equation by $\phi^{-1} \, \psi^2$ and integrate the result by part to  get, after some simple rearrangement, 
\[
	\int_{\mathbb R^{n+1}}  ( |\nabla \psi|^2  - \psi^2 + 3\, u^2 \,  \psi^2 ) \, {\rm dx}   = 
	\int_{\mathbb R^{n+1}}  \left| \nabla \psi  - \phi^{-1} \, \psi \, \nabla \phi \right|^2 \, 
	{\rm 	dx}  \geq 0 .
\]
Hence the solution $u$ is stable in the sense of Definition~\ref{de:111}. In the case where $u$ is monotone in the $x_{n+1}$ direction, this argument can be applied with $\phi  = \partial_{x_{n+1}} u$ to prove that monotone solutions of (\ref{ac}) are stable. 

As a consequence, in dimension $9$, the monotone solutions constructed by del Pino, Kowalczyk and Wei \cite{dkwdg} provide some non trivial stable solutions of the Allen-Cahn equation.

In the present paper, we show the~:
\begin{theorem}
\label{main1}
Assume that $n+1 = 2m \geq  8$. Then, there exist bounded, stable solutions of (\ref{ac}) whose level sets are not hyperplanes.
\end{theorem}
In fact, we can be more precise and we prove that the zero set of the stable solutions we construct are asymptotic to a minimal cone in $\mathbb R^{2m}$ defined by
\[
	C_{m,m} : = \{Ê(x,y) \in \mathbb R^m \times \mathbb R^m \, : \, |x| = |y| \} ,
\]
which is usually referred to as {\em Simons' cone}. The proof of Theorem~\ref{main1} strongly uses the fact that  {\em Simons' cone} is a minimal hypersurface which is {\em  minimizing} and hence is stable but also uses the fact that this cone is {\em strictly area minimizing} (we refer to Definition~\ref{de:4.1} and Definition~\ref{de:samc} for the precise definitions of these notions). 

Let us mention that, in dimension $n+1 = 2m$, with $m \geq 1$, Cabr\'e and Terra \cite{Cab-Ter-1} have found solutions of the Allen-Cahn equation whose zero set is exactly given by the cone $C_{m,m}$. When $m\geq1$, these solutions generalize the so called {\em saddle} solutions which have been found by Dang, Fife and Peletier \cite{Dan-Fif-Pel} in dimension $2$. The proof of this result makes use of a variational argument in the spirit of \cite{Dan-Fif-Pel}. Moreover, the same authors have proven that, when $m=2$ or $3$, the solutions they find is unstable \cite{Cab-Ter-2}. Cabr\'e has recently proven that saddle solutions are stable in dimension $2m \geq 14$ \cite{Cab}. 

Our result is in fact a corollary of a more general result~:
\begin{theorem}
\label{main2}
Assume that $C$ is a {\em minimizing cone} in $\mathbb R^{n+1}$ and that the indicial root $\nu_0^+ <0$ (see Definition~\ref{de:indroot}). Then, there exist bounded solutions of (\ref{ac}) whose zero sets are asymptotic to $C$ at infinity. 
\end{theorem}
As we will see, the solutions of (\ref{ac}) we construct are not unique and in fact they arise in families whose dimension can be computed. 

\begin{remark}
It should be clear that our construction, and hence the statement of Theorem~\ref{main2} does not reduce to the case of minimizing cone but extends to the case of minimal cones which are not necessarily minimizing but satisfy some natural nondegeneracy condition. Indeed, what is really needed in the proof of Theorem~\ref{main2} is the existence of a smooth minimal hypersurface $\Gamma$ which is asymptotic to the given minimal cone $C$. In the case where the minimal hypersurface is {\em nondegenerate}, in a sense to be made precise (see the last Remark in section~\ref{se:e10}), one should be able to modify the proof to construct solutions of (\ref{ac}) whose zero set is asymptotic to $C$.
\end{remark}

In view of the {\em de Giorgi's conjecture}, the results of Dancer and the above result, the following statement seems natural~: 

{\bf Classification of stable solution of the Allen-Cahn equation~:}  {\em Assume that $u$ is a bounded, stable solution of (\ref{ac}) in dimension less than or equal to $7$. Then, the level sets of $u$ should be hyperplanes.} 

This question parallels the corresponding well known conjecture concerning the classification of stable, embedded minimal hypersurface in Euclidean space~:   

{\bf Classification of stable, embedded minimal hypersurfaces~:} {\em The only stable, embedded minimal hypersurfaces in Euclidean space $\mathbb R^{n+1}$ are hyperplanes as long as the dimension of the ambient space is less than or equal to $7$.} 

This latter  problem is still open except when the ambient dimension is equal to $3$ where the result of Ficher-Colbrie and Schoen \cite{Fic-Col-Sch} guaranties that affine planes are the only stable embedded minimal surface in $\mathbb R^3$.

\section{Plan of the paper}

In the next section, we gather the necessary material to describe and analyze minimal cones. We then describe minimal hypersurfaces which are asymptotic to minimal cones. The specific case of strictly area minimizing cones is discussed in section 5. Next, in section 6, we introduce Fermi coordinates about a hypersurface and derive the expression of the Euclidean Laplacian in these coordinates. Section 7 is devoted to the definition of the approximate solution and some formal expansion of the solution we are looking for. In the next three sections, we derive the linear analysis which is relevant to our problem. We complete the proof of Theorem~\ref{main2} in section 10 where some fixed point argument is applied to solve a nonlinear problem. The last section of the paper, section 11, is concerned with the proof of Theorem~\ref{main1}.

\section{The geometry of minimal cones}
\label{se:3}

Assume that $\Lambda \subset S^{n}$ is a smooth, compact, oriented, minimal hypersurface which is embedded in $S^{n}$. The induced metric on $\Lambda$ will be denoted by $\bar g$ and the second fundamental  form will be denoted by $\bar h$. Recall that it is defined by
\[
	\bar h (t_1, t_2) : =  - g_{\circ} ( \nabla_{t_1}^{\circ} \bar N,  t_2) ,
\] 
for all $t_1, t_2 \in T_p \Lambda$. Here $\bar N$ is the normal vector field to $\Lambda$ in $S^n$, $g_\circ$ is the standard metric on $S^n$ an $\nabla^{\circ}$ denotes the covariant derivative in $S^n$.  It is also convenient to define the tensor $\bar h \otimes \bar h$ by the formula
\[
	\bar h \otimes \bar h (t_1, t_2) : =  g_{\circ} ( \nabla_{t_1}^{\circ} \bar N, 
	\nabla_{t_2}^{\circ} \bar N) ,
\]
for all $t_1, t_2 \in T_p \Lambda$.

Since we have assumed that the hypersurface $\Lambda$ is minimal in $S^n$, its mean curvature vanishes and hence 
\[
	{\rm Tr}_{\bar g} \, \bar h =0 .
\]
The  Jacobi operator about $\Lambda$ appears as the linearized mean curvature operator about $\Lambda$ when nearby hypersurfaces are parameterized as normal graphs over $\Lambda$. It is given by
\begin{equation}
\begin{array}{rlll}
	\bar J_\Lambda  & : = & \Delta_{\bar g}Ê+ \mbox{Tr}_{\bar g} (\bar h \otimes \bar h) + \mbox{Ric}_{g_\circ} ( \bar N, \bar N) \\[3mm]
	& = & \Delta_{\bar g}Ê+ {\rm Tr}_{\bar g}  (\bar h \otimes \bar h) +n-1 ,
\end{array}
\label{eq:jacop}
\end{equation}
where $\mbox{Ric}_{g_\circ}$ is the Ricci tensor of $S^n$ and ${\rm Tr}_{\bar g}  (\bar h \otimes \bar h)$ is the square of the norm of the second fundamental form (which, in the literature, is often denoted by $|A_\Lambda|^2$). 

Let us now give some important examples of minimal hypersurfaces in $S^n$ which are obtained as products of lower dimensional spheres. 
\begin{example}
We consider 
\[
	\Lambda_{n_1, n_2} := S^{n_1} (\rho_1) \times S^{n_2}(\rho_2) \subset S^n ,
\] 
where $n_1+ n_2 =  n-1$ and the radii $\rho_1$ and $\rho_2$ are chosen to be 
\[
	\rho_1 = \sqrt{\frac{n_1}{n-1}} \qquad \mbox{and} \qquad   \rho_2 : =  \sqrt{\frac{n_2}{n - 1}}.
\]
In this case, it is a simple exercise to check that the induced metric on $\Lambda_{n_1, n_2}$ is given by
\[
	\bar g  = \rho_1^2 \, g_{1} + \rho_2^2  \, g_{2} ,
\]
and that the second fundamental form (for some choice of the normal vector field) reads
\[
	\bar h  = \rho_1 \, \rho_2 \, \left( g_{1}  -  g_{2} \right) ,
\]
where $g_i$ is the induced metric on $S^{n_i}$. In particular, the hypersurface $\Lambda_{n_1, n_2}$ is minimal in $S^{n}$ since 
\[
	{\rm Tr}_{\bar g} \, \bar h = n_1 \, \frac{\rho_2}{\rho_1}  - n_2 \, \frac{\rho_1}{\rho_2}  = 0 . 
\]
Also, in this case, the Jacobi operator about $\Lambda_{n_1, n_2}$ is given by
\[
	\bar J_{n_1, n_2} =   \Delta_{\bar g}Ê +  2n -2 , 
\]
since 
\[
	{\rm Tr}_{\bar g} \, (\bar h \otimes \bar h) = \left(\frac{r_2}{r_1}\right)^2 \, n_1 + 
	\left(\frac{r_1}{r_2}\right)^2 \, n_2  = n_2 + n_1 = n-1.
\]
\label{ex:1}
\end{example}
This list of examples does not exhaust the list of all minimal hypersurfaces in $S^n$ and further examples of minimal hypersurfaces in $S^{n}$ can be found for example
in \cite{Law}, \cite{Pit-Rub} or in \cite{Kap-Seo}.

Given  an oriented embedded minimal hypersurface $\Lambda$ in $S^n$, we define 
\[
	C_\Lambda : = \{ e^t \, z   \, : \, z  \in  \Lambda  , \quad  t \in \mathbb R \}  \subset \mathbb R^{n+1},
\]
to be the cone over $\Lambda$. In this case, $\Lambda$ is usually referred to as the {\em link of the cone}. The induced metric on $C_\Lambda$ is given by
\begin{equation}
	g : =  e^{2t} \, (dt^2 + \bar g ).
\label{metriccone}
\end{equation}
The normal vector field on $C_\Lambda$ can be chosen in such a way that it coincides with the unit normal vector field of $\Lambda$ at each point of $C_\Lambda \cap S^{n}$. Observe that, when $n\geq 2$,  $S^n \setminus \Lambda$ has two connected components (this follows at once from the maximum principle which implies that $S^n$ does not contain two embedded minimal hypersurfaces which are disjoint) and hence $\mathbb R^{n+1} \setminus  C_\Lambda$ also has two connected components which will be denoted by $\mathbb R^{n+1}_\pm$. We can assume that $\mathbb R^{n+1}_+$  is the one toward which the normal vector field on $C_\Lambda$ is pointing. 

With these choices, the second fundamental form about $C_\Lambda$ reads
\begin{equation}
	h = e^t  \, \bar h. 
\label{metricconebis}
\end{equation}
Since we have assumed that $\Lambda$ is a minimal submanifold in $S^{n}$, we conclude that $C_\Lambda$ is also a minimal submanifold of $\mathbb R^{n+1}$. We also conclude that $J_C$, the Jacobi operator about $C := C_\Lambda$, is given by
\begin{equation}
	J_C  : = e^{-2t} \, \left( \partial_t^2 + (n-2) \, \partial_t  + \Delta_{\bar g}Ê + {\rm 
	Tr}_{\bar 	g} \, (\bar h \otimes \bar h)  \right).
\label{eq:jacopC}
\end{equation}

We will denote by 
\[
	\mu_0 < \mu_1 \leq \mu_2 \leq  \ldots
\] 
the eigenvalues of the operator 
\[
- (\Delta_{\bar g}Ê + {\rm Tr}_{\bar g} \, (\bar h \otimes  \bar h)) ,
\] 
and we will denote by $\varphi_j$ the eigenfunction which is associated to the eigenvalue $\mu_j$ and which is normalized to have $L^2(\Lambda)$ norm equal to $1$. Observe that the $\varphi_j$ are also the eigenfunctions of the Jacobi operator $-J_\Lambda$ associated to the eigenvalue $\mu_j+1$.  Also observe that $\mu_0 \leq 0$ since the potential ${\rm Tr}_{\bar g} \, (\bar h \otimes  \bar h)) \geq 0$.

We recall the following~:
\begin{definition}
The minimal cone $C_\Lambda$ is said to be {\em stable} if 
\[
	\mu_0 \geq \left(\frac{n-2}{2}\right)^2  , 
\]
and it is said to be {\em strictly stable} if
\[
	\mu_0 > \left(\frac{n-2}{2}\right)^2 .
\]
\label{de:3.1}
\end{definition}

Finally, we define the {\em characteristic roots} of the operator $J_C$ by the formula
\begin{equation}
	\gamma_j^\pm : =  \frac{2-n}{2} \pm \sqrt{Ê\left( \frac{n-2}{2}\right)^2 + 
	\mu_j}.
\label{eq:3-99}
\end{equation}
The characteristic roots appear in the asymptotic behavior of the solutions of  
\[
J_C w =0.
\]  
Indeed, looking for solutions of this equation of the form 
\[
w  = e^{\gamma t} \, \varphi_j ,
\]
one finds that $\gamma$ satisfies the characteristic equation
\[
\gamma^2 + (n-2) \, \gamma - \mu_j = 0 ,
\]
and hence $\gamma = \gamma_j^\pm$.  Observe that, in the special case where $\gamma_j^\pm = \frac{2-n}{2}$, then the two independent solutions of the above equation are of the form 
\[
w_j^+ =  e^{\frac{2-n}{2} t} \varphi_j, \qquad \mbox{and} \qquad  w_j^- = t \, e^{\frac{2-n}{2} t} \varphi_j.
\]

We have the important definition which is borrowed from \cite{Caf-Har-Sim}~:
\begin{definition}
The {\em indicial roots} of the operator $J_C$ are defined by
\begin{equation}
	\nu_j^\pm : = \Re \left( \frac{2-n}{2} \pm \sqrt{Ê\left( \frac{n-2}{2}\right)^2 + \mu_j}\right).
\label{eq:3-9}
\end{equation}
\label{de:indroot}
\end{definition}
The indicial roots of $J_C$ play a key role in the deformation theory of minimal cones and minimal hypersurfaces  which are asymptotic to minimal cones \cite{Caf-Har-Sim}, \cite{Cha} and, as we will see, they also play a very important role in our construction.  Observe that we always have 
\[
\nu_0^+ \leq 0
\]
since $\mu_O \leq 0$. 

We illustrate all these definitions in the case described in  Example~\ref{ex:1}. 
\begin{example}
Keeping the notations introduced in Example~\ref{ex:1}, we consider $C_{n_1, n_2}$ to be the cone over $\Lambda_{n_1, n_2}$. According to the above analysis,  this cone is a minimal hypersurface in $\mathbb R^{n+1}$. Moreover, the Jacobi operator about $C_{n_1, n_2}$ is given by
\[
	J_{n_1, n_2} : =  e^{-2t} \, \left( \partial_t^2 + (n-2) \, \partial_t  + \Delta_{\bar g}Ê 
	+ n-1 \right) .
\]
Surprisingly, this operator only depends on $n_1$ and $n_2$ through the metric $\bar g$ which is used to compute the Laplace-Beltrami operator. The indicial roots of $J_{n_1, n_2}$ can be computed explicitly in terms of the spectrum of the Laplace-Beltrami operator on the spheres $S^{n_1}$ and $S^{n_2}$. Of interest, will be the values of $\nu_0^\pm$ which are given by 
\begin{equation}
	\nu_0^\pm : = \Re \, \left( \frac{2-n}{2} \pm \frac{1}{2}Ê\, \sqrt{Ên^2 -8n + 8 } 
	\right).
\label{eq:gamma0}
\end{equation}
In particular, the cone $C_{n_1, n_2}$ is  {\em strictly stable} when $n+1 \geq 8$.

When $n+1 \leq 7$, we have
\[
\nu_0^+ = \nu_0^- = \frac{2-n}{2} ,
\]
while, when $n +1 \geq 8$, we have
\[
3-n < \nu_0^- \leq  4-n    \qquad \mbox{and} \qquad  - 2 \leq  \nu^+_0 <  -1.
\]
Finally, in the special case where $n+1=8$, we have 
\[
\nu_0^- =-3 \qquad \mbox{and} \qquad \nu^+_0 =-2 .
\]
\end{example} 

\section{Minimal hypersurfaces which are asymptotic to a minimal cone.}

In this section, we are interested in the existence and properties of minimal hypersurfaces which are smooth, embedded in $\mathbb R^{n+1}$ and which are asymptotic to a given minimal cone $C = C_\Lambda$ where $\Lambda$ is an embedded minimal hypersurface in $S^n$. The material of this section is essentially borrowed from  \cite{Har-Sim} and \cite{Cha}. 

As far as the existence of such a minimal hypersurface is concerned, we have a very general result but,  before, stating the result, we recall the definition of {\em minimizing cones}  \cite{Har-Sim}~:
\begin{definition}
\cite{Har-Sim} A minimal cone $C = C_\Lambda$ is said to be {\em minimizing} if 
\[
	{\rm Vol}_n ( C \cap B_1) \leq {\rm Vol}_n ( S ) ,
\]
for any hypersurface $SÊ\subset R^{n+1}$ such that $\partial S = \Lambda$, where $B_1$ denotes the unit ball in $\mathbb R^{n+1}$.
\label{de:4.1}
\end{definition}
In particular, a minimizing cone is necessarily stable in the sense of Definition~\ref{de:3.1}. 

As already mentioned, when $n \geq 2$, $\mathbb R^{n+1} \setminus C$ has two connected components which we denote by $\mathbb R^{n+1}_\pm$. When $C$ is a minimizing cone, the existence of smooth minimal hypersurfaces lying on one side of $C$, i.e. which are embedded in one of the connected components of $\mathbb R^{n+1} \setminus C$ follows from Theorem 2.1 in \cite{Har-Sim}. In fact, the authors prove that there are two families of such hypersurfaces which are embedded in the two different connected components of $\mathbb R^{n+1} \setminus C$.
\begin{theorem}
\cite{Har-Sim} Assume that $C$ is {\em minimizing cone}. Then, there exists two distinct oriented,  embedded minimal hypersurfaces 
\[
	\Gamma_\pm \subset \mathbb R^{n+1}_\pm, 
\] 
enjoying the following properties~:
\begin{itemize}
	\item[(i)] The minimal hypersurface $\Gamma_\pm$ sits on one side of $C$,  namely $\Gamma_+ $ (resp.  $\Gamma_- $) is embedded in $\mathbb R^{n+1}_+$ (resp. $\mathbb R^{n+1}_-$) ;
	\item[(ii)] The distance from $\Gamma_\pm$ to the origin is equal to $1$ ;
	\item[(iii)] The hypersurface $\Gamma_\pm$ is asymptotic to $C$, i.e. $\Gamma_\pm$ is a normal graph over $C$ for some function $v_\pm$ which tends to $0$ at infinity as a (negative)  power of the distance to the origin ;
	\item[(iv)] For all ${\bf e} \in \mathbb S^{n}\setminus \Lambda$, the half affine line $\mathbb R^+ {\bf e}$ meets $\Gamma_+ \cup \Gamma_-$ at, at most, one point. 	  
\end{itemize}
\label{th:4.1}
\end{theorem}
Observe that property (iv) implies that, the sets $\lambda \, \Gamma_\pm$ for $\lambda >0$ forms a foliation of $\mathbb R^{n+1}_\pm \setminus C$. Property (iii) implies that asymptotics of $\Gamma_\pm$ are well understood and in fact it is related to the indicial roots of $J_C$. This property together with the fact that $ \Gamma_\pm$ lie on one side of $C$ implies the following result which is also due to Hardt and Simon \cite{Har-Sim}~:
\begin{theorem}
\cite{Har-Sim} Assume that the cone $C$ is {\em stable} and that $\Gamma$ is a minimal hypersurface which is defined away from a compact of $\mathbb R^{n+1}$ and which, at infinity, is asymptotic to $C$ and lies on one side of $C$. Further assume that $\Gamma$ is a normal graph over the cone $C$ for a  function $v$, then either 
\[
	v (t,z) =  (b \, t + a) \, e^{\gamma_0^+ \, t} \, \varphi_0 (z) + {\mathcal O} 	
	(e^{ (\gamma_0^+ - 	\delta) \, t}), 
\]
or 
\[
	v (t,z) =  a  \, e^{\gamma_0^- \, t} \, \varphi_0 (z) + {\mathcal O} (e^{ (\gamma_0^- 
	- \delta) \, t}),
\]
for $z \in \Lambda$ and $t >0$ large enough. Here $\delta >0$,  $b=0$ and $a \neq 0$ unless $\gamma_0^\pm = \frac{2-n}{2}$  in which case $b \neq 0$. 
\label{th:4.2}
\end{theorem}
In other words, for minimal hypersurfaces which are asymptotic to a stable cone and are embedded on one side  of the cone, there are only two possible asymptotic behavior. 

In this statement, we recall that $\varphi_0$ is the eigenfunction of $- J_\Lambda$, the Jacobi operator about  $\Lambda$, which is associated to the first eigenvalue of $-J_\Lambda$. 

More examples of embedded minimal hypersurfaces which are asymptotic to a given minimal cone $C$ and which lie on one side of the cone $C$ can be found for example in \cite{Cha}.

In the case where 
\[
	\Lambda = \Lambda_{n_1, n_2} : =  S^{n_1} (r_1) \times S^{n_2} (r_2),
\]
is the minimal hypersurface of $S^n$ which is described in Example~\ref{ex:1}, the minimal cone $C_{n_1, n_2}$ is invariant under the action of a large group of symmetries, namely
\[
	O(n_1+1) \times O(n_2+1) ,
\]
and a parameterization of the minimal hypersurfaces $\Gamma_\pm$ can be obtained by solving some second order ordinary differential equation.  This is the point of view which is taken in \cite{Ale-Bar} where one looks for minimal  hypersurfaces which can be parameterized as 
\[
	\mathbb R \times S^{n_1} \times S^{n_2} \ni (s, z_1, z_2) \longrightarrow (x(s) \, 
	z_1 , y(s) \, z_2) \in \mathbb R^{n+1} .
\]
The fact that the mean curvature of this hypersurface is zero reduces to the ordinary differential equation
\[
	\frac{y'' \, x'- x'' \, y'}{(x')^2+ (y')^2} + n_1 \, \frac{y'}{x} - n_2 \, \frac{x'}{y}  = 0 ,
\]
and, without loss of generality, we can assume that the generating curve 
\[
	\mathbb R  \ni s  \longrightarrow (x(s) , y(s)) \in \mathbb R^{2} ,
\]
is  parameterized by arc length, namely
\[
	(x')^2 + (y')^2 = 1.
\]
where $'$ denotes the derivative with respect to $s$. As in \cite{Ale-Bar}, we define the functions $u$ and $v$ by the identities
\[
	\tan u = \frac{y}{x} \qquad \mbox{and} \qquad \tan v = \frac{y'}{x'}. 
\]
Then, one can check that the system of equations satisfies by $x$ and $y$ 
can also be written as 
\begin{equation}
	\left\{
	\begin{array}{rllll}
	u' & = & \cos u \, \sin u \, \sin (u-v) \\[3mm]
	v' & = & n_1 \, \sin u \, \sin v - n_2 \, \cos u \, \cos v
	\end{array}
	\right.
\label{eq:dynsyst}
\end{equation}
It is proven in \cite{Ale-Bar} that there exists a heteroclinic solution to this system that connects two stationary points of (\ref{eq:dynsyst}). This solution gives rise to an embedded minimal hypersurfaces which are asymptotic  to $C_{n_1, n_2}$. Moreover, in dimension $n \geq 7$, they prove that these hypersurfaces are one of the two connected components of $\mathbb  R^{n+1} \setminus C_{n_1, n_2}$ and that family of hypersurfaces $\lambda \, \Gamma$ forms a foliation of the connected components of $\mathbb  R^{n+1} \setminus C_{n_1, n_2}$, while, when $n \leq 6$, these hypersurfaces intersect $C_{n_1, n_2}$ infinitely many times.

\section{The case of strictly area minimizing cones}

In \cite{Har-Sim} is introduced the notion of {\em strictly area minimizing cone}.
\begin{definition}
\label{de:samc}
A minimal cone $C = C_\Lambda$ is said to be {\em strictly area minimizing} if there exists a constant $C >0$ such that, for all $\varepsilon >0$ small enough
\[
	{\rm Vol}_n ( C \cap B_1) \leq {\rm Vol}_n ( S)  - C \, \varepsilon^n,
\]
for any hypersurface $SÊ\subset R^{n+1}\setminus B_\varepsilon$ such that $\partial S = \Lambda$, where  $B_1$ denotes the unit ball and $B_\varepsilon$ the ball of radius $\varepsilon$  in $\mathbb R^{n+1}$.
\end{definition}

In the case where the cone $C = C_\Lambda$ is {\em strictly area minimizing}, we  have the following result \cite{Har-Sim} which states that the minimal hypersurfaces $\Gamma_\pm$ defined in Theorem~\ref{th:4.1} approach $C$ at the slowest possible rate predicted by Theorem~\ref{th:4.2}~:
\begin{proposition} 
Assume that the cone $C = C(\Sigma)$ is {\em strictly area minimizing}, then the minimal surface defined in Theorem~\ref{th:4.1} is, at infinity, a normal graph over the cone $C$ for a function $v$ which can be expanded as either 
\[
	v (t,z) =  a  \, e^{\gamma_0^+ \, t} \, \varphi_0 (z) + {\mathcal O} 	
	(e^{ (\gamma_0^+ - 	\delta) \, t}), 
\]
for some $a \neq 0$, if $\gamma_0^\pm \neq  \frac{2-n}{2}$ or 
\[
	v (t,z) =  (b \, t + a)  \, e^{\gamma_0^+ \, t} \, \varphi_0 (z) + {\mathcal O} 
	(e^{ (\gamma_0^+ - \delta) \, t}),
\]
for some $b \neq 0$, if $\gamma_0^\pm = \frac{2-n}{2}$. Here, $z \in \Lambda$, $t >0$ is large enough and $\delta >0$. 
\end{proposition}

Checking whether a minimal cone is {\em strictly area minimizing}  can be a hard problem.  For example, $\mathbb R^2 \subset \mathbb R^3$, which can be considered as a cone over the unit circle $S^1 \subset S^2$, is {\em area minimizing} but is not {\em strictly area minimizing} while, for $n \geq 3$, $\mathbb R^n \subset \mathbb R^{n+1}$, which can be considered as a cone over the unit sphere $S^{n-1} \subset S^n$, is {\em strictly area minimizing} \cite{Har-Sim}. Hopefully, there are a lot of minimal cone which are known to be {\em strictly area minimizing} \cite{Law}. For example, it is proven in \cite{Law} that $C_{n_1, n_2}$ are {\em strictly area minimizing} provided $n_1+ n_2  = n - 1 \geq 6$,  and this is also reflected in the analysis of (\ref{eq:dynsyst}) which is performed in \cite{Ale-Bar}. 

\section{Local coordinates near an embedded hypersurface and expression of the Laplacian}

\subsection{Local coordinates near a hypersurface} 
In this section, we assume that $n \geq 1$ and that $\Gamma$ is an oriented smooth hypersurface embedded in $ \mathbb R^{n+1}$ which is asymptotic to a minimal cone. We first define the Fermi coordinates about $\Gamma$ and then, we provide some asymptotic  expansion of the Euclidean Laplacian in Fermi coordinates about $\Gamma$.

We denote by $g_{e}$ the Euclidean metric in $\mathbb R^{n+1}$. We denote by $N$ the unit normal vector field on $\Gamma$ which defines the orientation of $\Gamma$ and we define 
\begin{equation}
	Z (y, z) : = y + z \, N (y) ,
\label{eer}
\end{equation}
where $y \in \Gamma$ and  $z \in \mathbb R$. The implicit function theorem implies that $Z$ is a local diffeomorphism  from a neighborhood of a point $(y, 0) \in \Gamma \times \mathbb R$ onto a neighborhood of $y \in \mathbb R^{n+1}$. 

Given $z \in \mathbb R$, we define $\Gamma_z$ by
\[ 
	\Gamma_z := \{ Z (y,z)   \in \mathbb R^{n+1} \, : \,  y \in \Gamma \} . 
\]
For $z$ small enough (depending on the point $y \in \Gamma$ where one is working), $\Gamma_z$ restricted to a neighborhood of $y$ is a smooth hypersurface which is referred to as the {\em hypersurface parallel to $\Gamma$ at height $z$}. 

If $X$ denotes a parameterization of a neighborhood of $y$ in $\Gamma$, we set
\[
	X_z  :=  X  + z \, X^* N , 
\]
which, for $z$ small enough, is a parameterization of a neighborhood of $Z(y,z)$ in $\Gamma_z$ (it will be convenient to agree that $f^* g = f \circ g$).

The next two results express the geometry of $\Gamma_z$ in terms of geometric objects defined  on $\Gamma$.  In particular, in terms of   $\mathring  g$  the induced metric on $\Gamma$,  $\mathring  h$  the second fundamental form on $\Gamma$, which is defined by
\[
	\mathring  h (t_1, t_2)  : =  - \mathring g ( \nabla_{t_1} N  , t_2 \, ),
\] 
for all $t_1, t_2 \in T\Gamma$, and in terms of the square of the second fundamental form which is the tensor defined by
\[
	\mathring   h  \otimes  \mathring  h  (t_1, t_2) : =  \mathring g (  \nabla_{t_1}  N ,  
	\nabla_{t_2} N ),
\]
for all $t_1, t_2 \in T\Gamma$. In local coordinates, we have 
\[
	( \mathring h \otimes \mathring h)_{ij} =  \sum_{a,b} \mathring h_{ia} \, \mathring g^{a b} 
	\, \mathring h_{b j} .
\]
With these notations at hand, we have the~:
\begin{lemma}
The induced metric $g_z$ on $\Gamma_z$ is given by
\[
	g_z  = \mathring g - 2 \, z \, \mathring h + z^2 \, \mathring h \otimes \mathring h .
\] 
\label{le:2.1}
\end{lemma}
\begin{proof}
We just need to compute the coefficients of the induced metric on $\Gamma_z$ in the parameterization given by $X_z$. We find
\[
	\begin{array}{rllll}
	\partial_{y_i} X_z \cdot \partial_{y_j} X_z =  \partial_{y_i} X \cdot \partial_{y_j} X
	+  z \, \left( \partial_{y_i} X \cdot \partial_{y_j} \tilde N   + \partial_{y_i} \tilde N 
	\cdot  \partial_{y_j} X \right) + z^2 \, \partial_{y_i} \tilde N \cdot \partial_{y_j} \tilde
	N  ,
	\end{array}
\]
where $\tilde N : = X^* N$. We can use the definition of $\mathring g$ and $\mathring h$, we conclude that
\[
	\begin{array}{rllll}
	\partial_{y_i} X_z \cdot \partial_{y_j} X_z  =  \mathring g_{ij}  
	-  2 \, z \, \mathring h_{ij}  +   z^2 \, (\mathring h \otimes \mathring h)_{ij} \, .
	\end{array}
\]
This completes the proof of the result. \end{proof}

Similarly,  the mean curvature $H_z$ of $\Gamma_z$ can be expressed in term of $z$, $\mathring g$ and $\mathring h$. We have the~:
\begin{lemma}
The following expansion holds
\[
	H_{z} = \sum_{j=0}^\infty {\rm Tr}_{\mathring g} ( \underbrace{ \mathring h \otimes \ldots \otimes 
	\mathring h }_{j \, {\rm times}}) \, z^j .
\]
\label{le:2.2}
\end{lemma}
It will be convenient to define 
\[
\mathring h^{(j)}  : = \underbrace{ \mathring h \otimes \ldots \otimes \mathring h }_{j \, {\rm times}} ,
\]
so that $\mathring h \otimes \mathring h =  \mathring h^{(2)}$.  Before we proceed with the proof, let us observe that we have the alternative formula
\[
 {\rm Tr}_{\mathring g} \,  \mathring h^{(j)}  = \sum_{i=1}^n \kappa_i^j ,
\]
where the $\kappa_j$ denote the principal curvature of $\Gamma$.

\begin{proof}
Recall that the mean curvature appears in the first variation of the volume form of parallel 
hypersurfaces. Hence, we have the formula 
\[
	H_z = - \frac{1}{\sqrt{\mbox{det} \,  g_z}} \, \frac{d}{dz} \sqrt{\mbox{det} \, g_z} .
\]
Now, we can always assume that, at the point where the computation is performed, $g_{ij} = \delta_{ij}$. In this case we can identify $\mathring h$ with a symmetric matrix $A$ and  $\mathring k$ with $A^2$. In particular, we  can write
\[
	H_z =  - \frac{1}{\mbox{det} \,  (I - z \, A)} \, \frac{d}{dz} \mbox{det} (I -  z \, A) ,
\] 
and the result follows from the formula 
\[
\frac{d}{dz} \mbox{det} (I - z A) = \mbox{Tr} \left( \sum_{j=0}^\infty z^j \, A^{j+1}\right)
\] 
where $A \in M_n (\mathbb R)$ and $z \in \mathbb R$ is small. 
\end{proof}

The following result is just Gauss's Lemma, it gives the expression  of the Euclidean metric on the domain of  $\mathbb R^{n+1}$ which is parameterized by $Z$. 
\begin{lemma}
We have 
\[
	Z^* \, g_{e} =  g_z + dz^2, 
\]
where $g_z$ is considered as a family of metrics on $T\Gamma$, smoothly depending on the parameter $z$ which belongs to a neighborhood of $0\in \mathbb R$.
\label{le:6-3}
\end{lemma}

Recall that the Laplace-Beltrami operator is given by 
\[
	\Delta_g = \frac{1}{\sqrt{ |g| }} \, \partial_{x_i} \left( g^{ij} \, \sqrt{ | g |} \, \partial_{x_j}
	\, \right) .
\]
in local coordinates. Therefore, in the neighborhood of $\Gamma$ parameterized by $Z$, the Euclidean Laplacian in $ \mathbb R^{n+1}$ can be expressed in Fermi coordinates by the (well-known) formula 
\begin{equation}
	\Delta_{\mathbb R^{n+1}}  = \partial^2_z  - H_z \, \partial_z +  \Delta_{g_z} ,
\label{Lapx}
\end{equation}
which follows at once from Lemma~\ref{le:6-3} and the above formula for the Laplace-Beltrami operator.

\section{Construction of an approximate solution}
\label{se:7}

In this section, we use the Fermi coordinates which have been introduced in the previous section and rephrase the equation we would like to solve in some neighborhood of $\Gamma$. We also build an approximate solution to (\ref{ac}) whose nodal set is equal to $\Gamma$. 

We define 
\[
u_\varepsilon : =  u_1 ( \cdot / \varepsilon)
\]
and 
\[
\dot u_\varepsilon (z) : =  u_1'  ( z/ \varepsilon) \qquad \ddot u_\varepsilon (z) : =  u_1'' ( z/ \varepsilon), 
\]
where $u_1$ is the solution of (\ref{eq:fp1}). One should be careful that, with these notations since we have 
\[
\partial_z u_\varepsilon (z)  = \frac{1}{\varepsilon} \, \dot u_\varepsilon (z) 
 \qquad \mbox{and} \qquad \partial_z \dot u_\varepsilon (z)  = \frac{1}{\varepsilon} \, \ddot u_\varepsilon (z).
\]

Here we agree that $\Gamma$ is a smooth minimal hypersurface which is embedded in $\mathbb R^{n+1}$ and we use the notations introduced in the previous section for the Fermi coordinates about $\Gamma$. 
 
Given any (sufficiently small and sufficiently smooth) function $\zeta$ defined on $\Gamma$, we define $ \Gamma_\zeta$ to be the normal graph over $\Gamma$ for the function $\zeta$. Namely
\[
	\Gamma_\zeta := \{ y + \zeta(y) \, N(y) \in \mathbb R^{n+1}  \, : \, y \in \Gamma\} . 
\]
This notation should not be confused with $\Gamma_z$ which is the hypersurface parallel to $\Gamma$ at height $z$. We keep the notations of the previous section and, in a tubular neighborhood of $
\Gamma$ we define the function $u$ by
\[
	Z^* u (y, z) = \bar u \left( y, z - \zeta (y) \right) .
\]
It will be convenient to denote by $t$ the variable 
\[
	t : = z - \zeta(y) .
\]
Using the expression of the Laplacian in Fermi coordinates which has been derived in 
(\ref{Lapx}), we find with little work that the equation we would like to solve can be 
rewritten as
\begin{equation}
	\begin{array}{rllll}
	\varepsilon^2 \Big[( 1 +  \|Êd \zeta \|^2_{g_z}  ) \partial_t^2 \bar u +
	 \Delta_{g_z} \bar u  -  \left( H_\zeta  + \Delta_{g_z} \zeta    \right) 
	 \partial_t \bar u  \hspace{35mm}  \\[3mm]
	- 2 \, ( d \zeta , d \partial_t v)_{g_z}  \Big]_{|z = t + \zeta}  +  \bar u  - \bar u^3 =0 ,
	\end{array}
\label{eq:rfe}
\end{equation}
for  $t >0$ close to $0$ and $y \in \Gamma$. Some comments are due about the  notations. In this equation, and the equations below, all computations of the quantities between the square brackets $[ \quad ]$ are performed using the metric $g_z$ defined in Lemma~\ref{le:2.1} and considering that $z$ is a parameter. Once this is done, we set $z = t + \zeta (y)$. 

We define
\[
	\bar u  (y,t)  : = u_\varepsilon (t) + v ( y ,t ) ,
\]
in which case, the equation (\ref{eq:rfe}) becomes
\[
	\mathfrak N (v ,\zeta ) = 0 \, ,
\]
where we have defined 
\begin{equation}
	\begin{array}{rllll}
	\mathfrak N ( v ,\zeta)  & : =  \displaystyle \Big[ \left( \varepsilon^2 ( \partial_t^2 + 
	\Delta_{g_z} ) +1 - 3 u_\varepsilon^2 \right)  \, v  - \varepsilon \, ( \Delta_{g_z} \zeta +
	H_z )  ( \dot u_\varepsilon +  \varepsilon \, \partial_t v ) \\[3mm]
	& \quad  +   \| d \zeta \|^2_{g_z} \, (\dot u_\varepsilon + \varepsilon^2 \, \partial^2_t v)  - 2\,\varepsilon^2 \, ( d \zeta , d \, 	\partial_t v)_{g_z}  \Big]_{| z = t + \zeta} +   v^3 + 3 \, u_\varepsilon \, v^2 .
	\end{array}
\label{eq:nrfe}
\end{equation}

When $v \equiv 0$ and $\zeta \equiv 0$, we simply have 
\begin{equation}
	\mathfrak N (0,0) =  \displaystyle - \varepsilon \,  H_t  \,  \dot u_\varepsilon ,
\label{eq:7-18}
\end{equation}
Also recall that
\begin{equation}
\begin{array}{rllll}
	H_t  = {\rm Tr}_{\mathring g} \, \mathring h^{(2)} \, t  + {\rm Tr}_{\mathring g} \, \mathring h^{(3)} \, t^2 + \mathcal O ( t^3) .
\end{array}
\label{eq:7-19}
\end{equation}
Observe that we have implicitely used the fact that $\Gamma$ is a minimal hypersurface and hence $ {\rm Tr}_{\mathring g} \,  \mathring h =0$. 

We now further assume that $\Gamma$ is asymptotic to a minimal cone and we define, for all $y \in \Gamma$,
\begin{equation}
	d_\Gamma (y) := \sqrt{1 + \mbox{dist}_{\mathring g} (y_0, y)^2} ,
\label{eq:distgamma}
\end{equation}
where $y_0$ is a given point in $\Gamma$ and $ \mbox{dist}_{\mathring g}$ denotes the intrinsic distance on $\Gamma$. Away from a compact, $d_\Gamma$ is equivalent to the intrinsic distance from $y_0$ to $y$ on $\Gamma$.   

Since  we assume that $\Gamma$ is asymptotic to a cone, the principal curvatures of $\Gamma$ are bounded by a constant times $1/ d_\Gamma $ and this implies that we have the pointwise estimate 
\begin{equation}
  | \nabla^k {\rm Tr}_{\mathring g} \, \mathring h^{(j)}|_{\mathring g} \leq C_{j,k}  \,  (d_\Gamma)^{- j -k} ,
\label{eq:7-20}
\end{equation}
for all $k \geq 0$, where $C_{j,k} >0$. Using this information, we have the~:
\begin{lemma}
For all $k, k' \geq 0$, there exists a constant $C_{k,k'} >0$ such that 
\begin{equation}
	 |   \nabla^{k'}  \, \partial_t^{k} \, \mathfrak N (0 ,0 ) |_{\mathring g} \leq C_{k,k'} \, \varepsilon^{2-k} \, (d_\Gamma)^{-2-k'} ,
\label{eq:7.21} 
\end{equation}
in the neighborhood of $\Gamma$ which is parameterized by $Z$. 
\label{le:7.1}
\end{lemma}

Given a function $f$ which is defined in $\Gamma \times \mathbb R$, we define $\Pi$ to be the $L^2$-orthogonal projection on $\dot u_\varepsilon$, namely
\[
\Pi (f)  : =  \frac{1}{\varepsilon  \, c} \, \int_{\mathbb R}  f(y,t) \, \dot u_\varepsilon  (t) \, dt ,
\]
where the normalization constant 
\[
c : =\frac{1}{\varepsilon} \,  \int_{\mathbb R}Ê\dot u_\varepsilon^2 (t) \, dt =  \int_{\mathbb R}Ê(u_1')^2 (t) \, dt .
\]
Of importance for us, will be the $L^2$-projection of $\mathfrak N (0,0)$ over $\dot u_\varepsilon$. The crucial observation is that 
\begin{equation}
\begin{array}{rllll}
\displaystyle \int_{\mathbb R} H_t \, \dot u_\varepsilon^2 \, dt  & =  &  \displaystyle \sum_{j=1}^\infty  \left( \int_{\mathbb R} t^j \, \dot u_\varepsilon^2 \, dt \right) \, {\rm Tr}_{\mathring g} \, \mathring h^{(j+1)}  \\[3mm]
& = & \displaystyle \sum_{k=1}^\infty  \varepsilon^{2k+1} \,  \left( \int_{\mathbb R} t^{2k} \, (u_1' )^2 \, dt  \right) \, {\rm Tr}_{\mathring g} \, \mathring h^{(2k+1)}  ,
\end{array}
\label{eq:parity}
\end{equation}
because of parity.  Using this property, we conclude that~:
\begin{lemma}
For all $k  \geq 0$, there exists a constant $C_{k } >0$ such that 
\begin{equation}
	 |  \nabla^{k}  \, \Pi  \, (\chi\, \mathfrak N (0 ,0 )) |_{ \mathring g} \leq C_{k} \, \varepsilon^{3} \,  (d_\Gamma)^{- 3 -k} \, ,
\label{eq:7.22} 
\end{equation}
in $\Gamma$. Here $\chi$ is a cutoff function which is identically equal to $1$ when $|t|Ê\leq c \, d_\Gamma (y)$ for some $c >0$ fixed small enough. 
\label{le:7.2}
\end{lemma}

\begin{remark}
Notice that, if we use the fact that $\Gamma$ is asymptotic to a minimal cone, then the principal curvatures of $\Gamma$ are bounded by a constant times $1/d_\Gamma$ and this implies that there exists $c >0$ such that the domain where $Z$ is a diffeomorphism contains the set of points for which $|t| \leq c \,  d_\Gamma$.
\label{re:difeo}
\end{remark}

The function $\bar u_\varepsilon$, which is defined by
\begin{equation}
Z^* \bar u_\varepsilon (y,t) : =  u_\varepsilon (t) ,
\label{eqves}
\end{equation}
in a neighborhood of $\Gamma$, will be used to define an approximate solution to our problem. 

\section{Analysis of the model linear operator}

In this section, we analyze the operator 
\begin{equation}
L_\varepsilon  : =  \varepsilon^2 \, \left( \partial_t^2 + \Delta_{\mathring g}  \right) + 1 - 3 \, u_\varepsilon^2 ,
\label{eq:12-1}
\end{equation}
acting on functions defined on the product space $\Gamma \times \mathbb R $, which is naturally endowed with the product metric 
\[
\mathring g + dt^2.
\] 

First, we will recall some standard injectivity result which is the key result in this analysis. Then, we will use this results to obtain an {\em a priori} estimate for solutions of $L_\varepsilon \, w = f$, when the functions $w$ and $f$ are defined in appropriate weighted spaces and satisfy some orthogonality condition. The proof of the {\em a priori} estimate is by contradiction. Finally, application of standard results in functional analysis will provide the existence of a right inverse for the operator $L_\varepsilon$ acting on some infinite codimension function space. 

\subsection{The injectivity result}

We collect some basic information about the spectrum of the operator
\begin{equation}
L_0 : = -\left( \partial_t^2 +  1 - 3\, u_1^2 \right) ,
\label{eq:12-2}
\end{equation}
which arises as the linearized operator of (\ref{eq:fp1}) about $u_1$ and which is acting on functions defined in $\mathbb R$. All the informations we need are included in the~:
\begin{lemma}
The spectrum of the operator $L_0$ is the union of the eigenvalue $\mu_0 =0$, which is associated to the eigenfunction
\[
w_0 (t) : = \frac{1}{\cosh^2(\frac{t}{\sqrt{2}})} ,
\]
the eigenvalue $\mu_1= \frac{3}{2}$, which is  associated to the eigenfunction
\[
w_1 (t)  :=  \frac{\sinh (\frac{t}{\sqrt{2}})}{\cosh^2 (\frac{t}{\sqrt{2}})} ,
\]
and the continuous spectrum which is given by $[2, \infty)$.
\label{le:12-1}
\end{lemma}
\begin{proof}
The fact that the continuous spectrum is equal to $[2, \infty)$ is standard. The fact that the bottom eigenvalue is $0$ follows directly from the fact that the equation for $u_1$ is autonomous and hence the function $u'_1 = \partial_t u_1$, which decays exponentially fast at infinity,  is in the $L^2$-kernel of $L_0$.  Since this function is positive, it has to be the eigenfunction associated to the lowest eigenvalue of $L_0$. Direct computation shows that $\mu_1$ is an eigenvalue of $L_0$ and, finally, it is proven in \cite{Nik-Uva} that $\mu_0 =0$ and $\mu_1 = 3/2$ are the only eigenvalues of $L_0$.
\end{proof}

Observe that this result implies that the quadratic form associated to $L_0$ is definite positive when acting on functions which are $L^2$-orthogonal to $u_1'$. More precisely, we have the inequality
\begin{equation}
\int_{\mathbb R} \left( |\partial_t w|^2 - w^2 + 3 \, u_1^2 \, w^2\right) \, dt \geq \frac{3}{2} \, \int_{\mathbb R} w^2 \, dt ,
\label{eq:12-4}
\end{equation}
for all function $w \in H^1(\mathbb R)$ satisfying the orthogonality condition 
\begin{equation}
\int_{\mathbb R} w(t) \, u_1'(t) \, dt =0 .
\label{eq:12-5}
\end{equation}

As already mentioned, the discussion to follow is based on the understanding of the bounded kernel of the operator 
\begin{equation}
L_* : = \partial_t^2  + \Delta_{\mathbb R^n} + 1- 2 u_1^2 ,
\label{eq:12-6}
\end{equation}
which is acting on functions defined on the product space $\mathbb R \times \mathbb R^n$. This is the contain of the following~:
\begin{lemma}
Assume that $w \in L^\infty (\mathbb R \times \mathbb R^n )$ satisfies $L_* \, w =0$. Then $w$ only depends on $t$ and is collinear to $u_1'$.
\label{le:12.2}
\end{lemma}
\begin{proof}
The original proof of this Lemma, which is based on Fourier transform in $\mathbb R^n$, can be found in  \cite{Pac-Rit}. For the sake of completeness, we give here the proof which is more in the spirit of the proof in \cite{dkwdg}. First, we observe that, by elliptic regularity theory, the function $w$ is smooth and we can decompose
\[
w(y,t) =  c(y) \, u_1'(t) + \bar w (t,y) ,
\]
where $\bar w (y, \cdot )$ satisfies (\ref{eq:12-5}) for all $y \in \mathbb R^n$. Inserting this decomposition into the equation satisfied by $w$, we find 
\[
u_1'  \, \Delta_{\mathbb R^n} \, c + \left( \partial_t^2  + 1- 2 u_1^2 \right)  \, \bar w + \Delta_{\mathbb R^n} \, \bar w = 0 .
\]
Multiplying this equation by $u_1'$ and integrating the result over $t \in \mathbb R$, we conclude easily that
\[
\Delta_{\mathbb R^n} c = 0,
\]
since $L_0 \, u'_1 =0$ and since $\Delta_{\mathbb R^n} \, \bar w$ is $L^2$-orthogonal to the function $u'_1$. By assumption $w$ is a bounded function and hence so is the function $c$. In particular, this implies that $c$ is the constant function. 

Next, we prove that $\bar w \equiv 0$. Since we have proven that $c$ is the constant function, we can now write 
\begin{equation}
\left( \partial_t^2  + 1- 2 u_1^2 \right)  \, \bar w + \Delta_{\mathbb R^n} \, \bar w =0
\label{eq:12-7}
\end{equation}

We claim that,  for any $\sigma \in (0, \sqrt 2)$, the function $\bar w$ is bounded by a constant times $(\cosh s)^{-\sigma}$. Indeed, in the equation (\ref{eq:12-7}), the potential, which is given by $1- 3u_1^2$, tends to $-2$ as $|t|$ tends to $\infty$. Using this property, one can check that, for all $\eta >0$ and $\delta \in (0, 1)$,  the function 
\[
W (y,t) : = e^{-\sigma |t|} + \eta  \,  \cosh (\delta t) \, \sum_{i=1}^n  \cosh (\delta y_i) , 
\] 
satisfies $L_* W  < 0$ in the region where  $|t| \geq t_*$, provided $t_* >0$ is fixed large enough (depending on $\sigma$).  Since $\bar w$ is bounded, we conclude that  
\[
|\bar w| \leq \|\bar w\|_{L^\infty} \, e^{\sigma t_*} \, \left( e^{-\sigma |t|} + \eta \,  \cosh (\delta t) \, \, \sum_{i=1}^n  \cosh (\delta y_i)\right),
\] 
when  $|t| \geq t_*$. Letting $\eta$ tend to $0$, this implies that 
\[
|\bar w| \leq  \|\bar w\|_{L^\infty}  \, e^{-\sigma (|t|-t_*)} ,
\] 
for $|t| \geq t_*$ and this completes the proof of the claim.  

Multiplying the equation satisfied by $\bar w$ by $\bar w$ intself and integrating the result over $\mathbb R$ (and not over $\mathbb R^n$), we find that 
\[
\int_{\mathbb R}  \left( | \partial_t \bar w|^2 - \bar w^2 + 3\, u_1^2 \, \bar w^2\right) \, dt  + \int_{\mathbb R} \bar w \, \Delta_{\mathbb R^n} \bar w \, dt  = 0 .
\]
Using the identity
\[
2 \, \bar w \, \Delta_{\mathbb R^n} \bar w = \Delta_{\mathbb R^n} \bar w^2  - 2 \, | \nabla w |^2, 
\]
together with Lemma~\ref{le:12-1}, we conclude that the function
\[
V  (y) : = \int_{\mathbb R} \bar w^2  (y,t) \, dt , 
\]
satisfies
\[
\Delta_{\mathbb R^n} V - \frac{3}{4} \, V  =  \int_{\mathbb R}  |\nabla \bar w|^2 \, dt \geq 0 .
\]

Let $E_1$ be the first eigenvalue of $-\Delta_{\mathbb R^n}$ in the ball of radius $1$, with $0$ Dirichlet boundary condition. The associated eigenvalue will be denoted by $\lambda_1$. In particular 
\begin{equation} 
- \Delta_{\mathbb R^n} E_1 =  \lambda_1 \, E_1.
\label{eq:12.7}
\end{equation}
Then $E_R(x) := E_1 (x/R)$ is the first eigenfunction of  $-\Delta_{\mathbb R^n}$ in the ball of radius $R$, with $0$ Dirichlet boundary condition, and the associated eigenvalue is given by $\lambda_1 / R^{2}$.

We multiply (\ref{eq:12-7}) by $E_R$ and integrate by parts the result over $B_R$, the ball of radius $R$ in $\mathbb R^n$. We get
\[
\left( \frac{\lambda_1}{R^{2}} -\frac{3}{4}Ê\right) \, \int_{B_R}  V \, E_R \, dx  + \int_{\partial 
B_R} \partial_r E_R \, V \, da  \geq  0 .
\]
Choosing $R$ large enough and using the fact that $V \geq 0$, we conclude that  $V \equiv 0$ in $B_R$. Therefore $V \equiv 0$ on $\mathbb R^n$.
\end{proof}

\subsection{The {\em a priori} estimate}

We are now in a position  to analyze the operator $L_\varepsilon$ which has been defined in (\ref{eq:12-1}), when it  is acting on H\"older weighted spaces which we now define.  We need to introduce some notations. Recall that $d_\Gamma$ denotes the function
\[
d_\Gamma (y) := \sqrt {1 + \mbox{\rm dist}_{\mathring g} (y, y_0)^2} ,
\]
where  $y_0 \in \Gamma$ is a given point in $\Gamma$. In other words, away from a compact, $d_\Gamma$ is equivalent to the intrinsic distance function to $y_0$. 

Next, we define on $\Gamma \times \mathbb R$, the scaled metric 
\[
g_\varepsilon := \varepsilon^2 \, (\mathring g + dt^2) .
\]
Given a point $x = (y, t) \in \Gamma \times \mathbb R $, we define $\|Êw\|_{\mathcal C^{k, \alpha}_{g_\varepsilon} ( B_\varepsilon (x , \varepsilon )) }$ to be the $\mathcal C^{k, \alpha}$ norm of the function $w$ in $ B_\varepsilon (x , \varepsilon )$, the geodesic ball of radius $\varepsilon$ centered at the point $x$, when the underlying manifold $\Gamma \times \mathbb R$ is endowed with the scaled metric $g_\varepsilon$.  With these notations in mind, we can state the~:
\begin{definition}
For all $k \in \mathbb N$, $\alpha \in (0,1)$ and $\nu \in \mathbb R$, the space $ \mathcal C^{k, \alpha}_{\varepsilon, \nu} (\Gamma \times \mathbb R )$ is the space of functions $ w\in \mathcal C^{k, \alpha}_{\rm loc} (\Gamma \times \mathbb R  )$ for which the following norm
\[
	\| w \|_{ \mathcal C^{k, \alpha}_{\varepsilon,  \nu } (\Gamma \times \mathbb R  )} : 
	=  \sup_{ x =  (y,t) \in \Gamma \times \mathbb R }  \, d_\Gamma(y)^{- \nu} \, \|Êw\|_{\mathcal C^{k, \alpha}_{g_\varepsilon} ( B_\varepsilon ( x, \varepsilon )) },
\]
if finite. 
\end{definition}
In other words, if $w \in  \mathcal C^{k, \alpha}_{\varepsilon, \nu} (\Gamma \times \mathbb R  )$, then 
\[
	d_\Gamma (y)^{-\nu} \,  | w( y,t) | \leq \| w\|_{ \mathcal C^{k, \alpha}_{\varepsilon, \nu} 
	(\Gamma \times \mathbb R )}  ,
\]
with similar estimates for the partial derivatives of $w$ when $\mathbb R \times \Gamma$ is endowed with the scaled metric $g_\varepsilon$. In particular,  
\[
	 | \nabla^a \, \partial_t^b  w(y,t) |_{\mathring g} 
	\leq C \, \| w\|_{ \mathcal C^{k, \alpha}_{\varepsilon, \nu } ( \Gamma \times 
	\mathbb R  )} \,  \varepsilon^{-a-b}  \, d_\Gamma (y)^{\nu} \,  , 
\]
provided  $a+b \geq k$. In other words, taking partial derivatives, we loose powers of $\varepsilon$ while the asymptotic behavior of the functions and its partial derivatives remains the same as $d_\Gamma (y)$ tend to $\infty$.

We shall work in the closed subspace of functions satisfying the orthogonality condition
\begin{equation}
	\int_{\mathbb R} \, w (y,t)  \, u'_1 (t)  \, dt = 0 \quad \mbox{for all} \quad y \in \Gamma .
\label{eq:12-10}
\end{equation}
We have the following~:
\begin{proposition}
Assume that  $\nu \in \mathbb R$ is fixed. Then, there exist  constants $C >0$ and $\varepsilon_0 >0$ such that, for all $\varepsilon \in (0, \varepsilon_0)$ and for all $w \in \mathcal C^{2, \alpha}_{\varepsilon, \nu } (\Gamma \times \mathbb R  )$ satisfying (\ref{eq:12-10}), we have 
\begin{equation}
	\| w \|_{\mathcal C^{2, \alpha}_{\varepsilon, \nu } (\Gamma \times \mathbb R  )} \leq C \,  \| L_\varepsilon \, 
	w\|_{\mathcal C^{0, \alpha}_{\varepsilon, \nu} (\Gamma \times \mathbb R  )} .
\label{eq:12-100}
\end{equation}
\label{pr:apriori}
\end{proposition}
\begin{proof}
Observe that, by elliptic regularity theory, it is enough to prove that 
\[
	\|  (d_\Gamma)^{-\nu}   \, w \|_{L^\infty (\Gamma \times \mathbb R )} \leq C  \, \|   (d_\Gamma)^{-\nu}  \, 
	  L_\varepsilon \, w\|_{ L^\infty (\Gamma \times \mathbb R  )} .
\]
The proof of this inequality is by contradiction. We assume that, for a sequence $\varepsilon_i$ tending to $0$ there exists a function $w_i$ such that 
\[
	\|  (d_\Gamma )^{-\nu} \, w_i \|_{L^\infty (\Gamma \times \mathbb R  )}  = 1,
\]
and 
\[
	\lim_{i \rightarrow \infty} \|   d_\Gamma (y)^{-\nu}  \, L_{\varepsilon_i} \, w_i \|_{ L^\infty (\Gamma \times \mathbb R  )}  =0
\]
For each $i \in \mathbb N$, we choose a point $x_i : = (y_i, t_i) \in \Gamma \times \mathbb R$ where 
\[
	 d_\Gamma (y_i)^{-\nu}   \, | w_i (y_i, t_i)| \geq 1/2.
\]
Arguing as in the proof of Lemma~\ref{le:12.2}, one can prove that the sequence $t_i$ tends to $0$ and more precisely that $|t_i| \leq C \, \varepsilon_i$. Indeed, the constant function can be used as a super-solution to show that necessarily $|t_i| \leq t_* \, \varepsilon_i$. 

Now, we use
\[
y \in T_{y_i} \Gamma \longmapsto {\rm Exp_{y_i}^\Gamma  (y)} \in \Gamma ,
\] 
the exponential map on $\Gamma$, at the point $y_i$, to define 
\[
\tilde w_i (y,t) : =   d_\Gamma (y_i)^{-\nu} \, w_i ( {\rm Exp_{y_i}^\Gamma  (\varepsilon_i \, y)}, \varepsilon_i \, t),
\]
which is defined on $ T_{y_i} \Gamma \times \mathbb R$.

Using elliptic estimates together with Ascoli's Theorem, we can extract subsequences and pass to the limit in the equation satisfied by $\tilde w_i$. We find that (up to a subsequence) $\tilde w_i$ converges, uniformly on compacts to $\tilde w$ which is a non trivial solution of
\[
\left (\partial_t^2 + 1 - 3 \, u_1^2 + \Delta_{\mathbb R^n}\right) \tilde w =0 ,
\]
in $\mathbb R^n \times \mathbb R$. In addition $\tilde w$ satisfies (\ref{eq:12-5}) and $\tilde w \in L^\infty ( \mathbb R^n \times \mathbb R)$. Since this clearly contradicts the result of Lemma~\ref{le:12.2}, the proof of the result is therefore complete.
\end{proof}

For all $R >0$, we denote by $B_R$ the geodesic ball of radius $R$ in $\Gamma$, centered at $y_0$. Using  arguments which are similar to the arguments used in the previous proof, one can prove the following result~:
\begin{proposition}
Assume that  $\nu \in \mathbb R$ is fixed. Then, there exist  constants $C >0$ and $\varepsilon_0 >0$ such that, for all $\varepsilon \in (0, \varepsilon_0)$, for all $R >1$ and all $w \in L^\infty  (B_R \times \mathbb R )$ which vanishes on $\partial B_R \times \mathbb R$ and which satisfies (\ref{eq:12-10}), we have 
\begin{equation}
\| (d_\Gamma)^{-\nu} \, w \|_{L^\infty  (B_R \times \mathbb R  )} \leq C \,  \|  (d_\Gamma )^{-\nu} \, L_\varepsilon \, w\|_{L^\infty (B_R \times \mathbb R  )} .
\label{eq:12-100}
\end{equation}
\label{pr:aprioririi}
\end{proposition}
\begin{proof}
The proof of this result is similar to the proof of the previous result. There is though one extra case to consider in the argument by contradiction : the case where the limit problem is defined in a half space. But this case can be ruled out by extending the function $\tilde w$ in the whole space by odd reflection. 

Indeed, keeping the notations of the previous proof, we now also need to consider the case where the distance from the point $y_i$ to the boundary of $\partial B_{R_i} \times \mathbb R$ is of order $\varepsilon_i$ (observe that elliptic estimates imply that this distance cannot be much smaller than $\varepsilon_i$ since the functions are assumed to vanish on $\partial B_{R_i} \times \mathbb R$, and hence, their gradient is controlled in a neighborhood of $\partial B_{R_i} \times \mathbb R$). In this case, up to a rigid motion, the limit problem is again 
\[
\left (\partial_t^2 + 1 - 3 \, u_1^2 + \Delta_{\mathbb R^n}\right) \tilde w =0,
\]
but, this time, the function $\tilde w$ is defined on $\mathbb R^{n}_+ \times \mathbb R$, where $\mathbb R^{n}_+ : = \{ (x_1, \ldots, x_n) \in \mathbb R^n \, : \, x_n >0\}$ and $\tilde w$ vanishes on $\partial \mathbb R^{n}_+ \times \mathbb R$.  Extending the function $\tilde w$ to all $\mathbb R^n \times \mathbb R$ by odd reflection, we reduce the problem to a case which is already studied in the previous proof. Details are left to the reader. 
\end{proof}

\subsection{The surjectivity result}

The final result of this section is the surjectivity of the operator $L_\varepsilon$ acting on the space of functions satisfying (\ref{eq:12-10}). 
\begin{proposition}
Assume that  $\nu  < 0$ is fixed. Then, there exists $\varepsilon_0 >0$ such that, for all $\varepsilon \in (0, \varepsilon_0)$ and for all $f \in  \mathcal C^{0, \alpha}_{\varepsilon, \nu}  (\Gamma  \times \mathbb R )$ satisfying (\ref{eq:12-10}), there exists a unique function $w \in \mathcal C^{2, \alpha}_{\varepsilon, \nu} (\Gamma  \times \mathbb R)$ which also satisfies (\ref{eq:12-10}) and which is a solution of 
\[
L_\varepsilon \, w  =  f ,
\]
in $\Gamma  \times \mathbb R$. 
\label{pr:8.3}
\end{proposition}
\begin{proof}
To begin with, we assume that $f$ has compact support and, for all $R >1$ large enough, we use the variational structure of the problem and consider the functional
\[
F (w) : =  \int_{B_R \times \mathbb R}  \left( \varepsilon^2 (|\partial_t w|^2 + |\nabla w|^2_{\mathring g}) - w^2 + 3 \, u_1^2 \, w^2\right)  \, {\rm dvol}_{\mathring g} \, dt ,
\]
acting on the space of functions $w \in H^1_0 (B_R \times \mathbb R)$ which satisfy (\ref{eq:12-10}) for a.e. $y \in \Gamma$. 

Thanks to Lemma~\ref{le:12-1}, we know that 
\[
F ( w ) \geq  \frac{3}{2} \,  \int_{B_R \times \mathbb R}  w^2 \,  {\rm dvol}_{\mathring g}\, dt.
\]
Now, given $f \in L^2 (B_R \times \mathbb R)$, we can apply Lax-Milgram's Theorem to obtain a weak solution of $L_\varepsilon \, w  = f$ in $H^1_0 (B_R \times \mathbb R)$.  

Assuming that $f \in L^\infty (\Gamma \times \mathbb R)$ has compact support, we make use  of the result of Proposition~\ref{pr:aprioririi} to get an {\em a priori} estimate for the solution which is defined in $B_R \times \mathbb R$ and standard arguments allows one to pass to the limit  as $R$ tends to $\infty$ to obtain the desired solution defined on $\Gamma \times \mathbb R$. Once the proof is complete when the function $f$ has compact support, the proof for general $f$ follows at once using an exhaustion by a sequence of functions $f_i$ which have compact support and converge uniformly to $f$ on compacts. 
\end{proof}

\section{Study of a strongly coercive operator}

This short section is devoted to the mapping properties of the operator 
\begin{equation}
\mathcal L_\varepsilon : =  \varepsilon^2 \, \Delta - 2 .
\label{eq:seo}
\end{equation}
Certainly this operator satisfies the maximum principle and solvability of the equation $\mathcal L_\varepsilon \, w =f$ and obtention of the estimates boils down to the construction of appropriate super-solutions.  we define 

Given a point $x \in \mathbb R^{n+1}$, we define $ \|Êw\|_{\mathcal C^{k, \alpha}_{\varepsilon^2 g_e} ( B_{\varepsilon} (x , \varepsilon )) },$ to be the $\mathcal C^{k, \alpha}$ norm of the function $w$ in $B_{\varepsilon} (x , \varepsilon )$, the geodesic ball of radius $1$ centered at the point $x$, when $\mathbb R^{n+1}$ is endowed with the scaled metric $\varepsilon^2 \, g_e$.  With these notations in mind, we can state the~:
\begin{definition}
For all $k \in \mathbb N$, $\alpha \in (0,1)$ and $\nu \in \mathbb R$, the space $ \mathcal C^{k, \alpha}_{\varepsilon, \nu} (\mathbb R^{n+1})$ is the space of functions $ w\in \mathcal C^{k, \alpha}_{\rm loc} (\mathbb R^{n+1} )$ for which the following norm
\[
	\| w \|_{ \mathcal C^{k, \alpha}_{\varepsilon,  \nu } (\mathbb R^{n+1} )} : 
	=  \sup_{ x  \in \mathbb R^{n+1} }  \, (1 + |x|^2)^{-\nu/2} \, \|Êw\|_{\mathcal C^{k, \alpha}_{\varepsilon^2 g_e} ( B_\varepsilon ( x, \varepsilon )) },
\]
if finite. 
\end{definition}

Some comments are due concerning the notations. The function space $ \mathcal C^{k, \alpha}_{\varepsilon, \nu} (\mathbb R^{n+1})$ should not be confused with the function space  $\mathcal C^{k, \alpha}_{\varepsilon, \nu} (\Gamma \times \mathbb R)$ even though, in spirit, the definition of their norms are very similar.

We have the~:
\begin{proposition}
Assume that $\nu \in \mathbb R$ is fixed. There exists a constant $C >0$ such that 
\begin{equation}
	\|  w \|_{\mathcal C^{2, \alpha}_{\varepsilon , \nu} ( \mathbb R^{n+1})} \leq C \,  \|  \mathcal L_\varepsilon \,  w\|_{\mathcal C^{0, \alpha}_{\varepsilon , \nu} (\mathbb R^{n+1})} ,
\end{equation} 
provided $\varepsilon \in (0,1)$. 
\label{pr:10.91}
\end{proposition}
\begin{proof}
Away from a compact, the function $x \longmapsto |x|^{\nu}$ can be used as a barrier to prove that
\[
	\|   (1 + |x|^2)^{-\nu/2} \, w \|_{L^\infty ( \mathbb R^{n+1})} \leq C \,  \|  (1 + |x|^2)^{-\nu/2} \,  \mathcal L_\varepsilon \,  w \|_{L^\infty (\mathbb R^{n+1})} ,
\]
for some constant $C >0$ which is independent of $\varepsilon$. The estimate is then a consequence of standard elliptic estimates applied on geodesic balls of radius $\varepsilon$.
\end{proof}

\section{The study of the Jacobi operator}
\label{se:10}
In this section, we introduce appropriate functions spaces in which the equation 
\[
J_\Gamma \, \zeta = \xi,
\]
can be solved for some function $\xi$ defined on $\Gamma$. 

We recall that, by definition 
\[
d_\Gamma (y)  : =  \sqrt{ 1+ {\rm dist}_{\mathring g}Ê(y, y_0)^2} ,
\]
where $y_0$ is a given point in $\Gamma$. 

\begin{definition}
For all $k \in \mathbb N$, $\alpha \in (0,1)$ and $ \nu \in \mathbb R$, the space $\mathcal C^{k, \alpha}_{\nu} (\Gamma )$ is the space of functions $ w\in \mathcal C^{k, \alpha}_{\rm loc} (\Gamma )$ for which the following norm
\[
\| \zeta \|_{ \mathcal C^{k, \alpha}_{\nu } (\mathbb R \times \Gamma )} : =  \sup_{ y \in  \Gamma} \left(  d_\Gamma (y)^{- \nu} \, \|Ê\zeta \|_{\mathcal C^{k, \alpha}_{\mathring g} ( B ( y, 1 )) } \, \right) ,
\]
if finite. Here, $B ( y, 1 )$ denotes the geodesic ball of radius $1$ in $\Gamma$, centered at $y$.
\end{definition}

Obviously
\[
J_\Gamma : \mathcal C^{2, \alpha}_\nu (\Gamma) \longrightarrow
\mathcal C^{0, \alpha}_{\nu-2} (\Gamma).
\]

The mapping properties of the operator $J_\Gamma$, defined between these weighted spaces, are intimately related to the indicial roots which have been defined in (\ref{eq:3-9}).

\begin{proposition}
Assume that $\nu \neq \gamma_j^{\pm}$, for all $j \in \mathbb N$ and define 
\[
\nu' = n - 2 + \nu .
\]
Then, the operator 
\[
J_\Gamma : \mathcal C^{2, \alpha}_\nu (\Gamma) \longrightarrow \mathcal C^{0, \alpha}_{\nu-2} (\Gamma),
\]
is {\em injective}, if and only if the operator 
\[
J_\Gamma : \mathcal C^{2, \alpha}_{\nu'} (\Gamma) \longrightarrow
\mathcal C^{0, \alpha}_{\nu' - 2} (\Gamma) ,
\]
is {\em surjective}.
\label{pr:10.10}
\end{proposition}
This proposition is by now standard. We refer to \cite{Caf-Har-Sim}, \cite{Cha}, \cite{Har-Sim} and also to \cite{Pac} for a proof.
 
In the cases of interest, namely, when the minimal hypersurface $\Gamma$ is asymptotic to a minimizing cone, we have the~:
\begin{lemma}
Assume that $\Gamma$ is a minimal surface which is asymptotic and lies on one side of a minimizing cone $C$. Then $J_\Gamma$ is injective in $\mathcal C^{2, \alpha}_\nu (\Gamma)$ for all $\nu < \nu_0^-$. If in addition, the cone is strictly area minimizing then the operator is injective for all $\nu < \nu_0^+$. 
\label{le:1010}
\end{lemma}
\begin{proof}
Just use the result of Theorem~\ref{th:4.1} together with the fact that that 
\[
\zeta_0 (y) : =  y \cdot N(y) ,
\]
is a Jacobi field (and hence solves $J_\Gamma \, \zeta_0 =0$) which is associated to the fact that the minimal surface equation is invariant under dilations. Observe that, thanks to (iv) in Theorem~\ref{th:4.1}), the function $\zeta_0$ does not change sign. According to the result of Theorem~\ref{th:4.2}, this Jacobi field does not belong to $\mathcal C^{2, \alpha}_{\nu} (\Gamma)$ for $\nu < \nu^-_0$, when the cone $C$ is minimizing and this Jacobi field does not belong to $\mathcal C^{2, \alpha}_{\nu} (\Gamma)$ for $\nu < \nu^+_0$ when the cone $C$ is strictly area minimizing. 

Since this Jacobi field does not change sign,  it can be used as a barrier to prove injectivity in the corresponding spaces. 
\end{proof}

\section{The nonlinear scheme}

We describe in this section the nonlinear scheme we are going to use to perturb an infinite dimensional family of approximate solutions into a genuine solution of (\ref{ac}). First, we define some cutoff functions which are used both in the definition of the approximate solutions and in the nonlinear scheme. Next, we define an infinite dimensional family of diffeomorphisms  which are used to construct the approximate solutions. Finally, we explain the nonlinear scheme we use. The last section is concerned with the solvability of the nonlinear problem and builds upon all the analysis we have done so far. 

From now on, we assume that $C$ is a minimizing cone and that the indicial root $\nu_0^+ <0$. We define $\Gamma$ to be the minimal hypersurface which is described in Theorem~\ref{th:4.2}. In particular, $\Gamma$ lies on one side of $C$ and Lemma~\ref{le:1010} applies. 

\subsection{\bf Some useful cutoff functions} We will need various cutoff functions in our construction. Therefore, for $j= 1, \ldots, 5$, we define the cut-off function $\chi_j$ by
\[
Z^* \chi_j (y,t) : =  
\left\{
\begin{array}{llll}
1 \qquad \mbox{when} \qquad |t|Ê\leq \varepsilon^{\delta_*} \, \left( d_\Gamma (y) - \frac{2j-1}{100} \right) \\[3mm] 
0 \qquad \mbox{when} \qquad |t|Ê\geq \varepsilon^{\delta_*} \, \left( d_\Gamma (y) - \frac{2j-2}{100}  \right),
\end{array}
\right.
\]
where $\delta_* \in (0, 1)$ is fixed. When $\varepsilon$  is chosen small enough, $Z$ is a diffeomorphism from the set $\{ (y,t) \in \Gamma \times \mathbb R \, : \,  |t|Ê\leq \varepsilon^{\delta_*} \, d_{\Gamma}\}$ onto its image.  We define $\Omega_j$ to be the support of $\chi_j$. By construction, $\Omega_j$ is included in the set of points where $\chi_{j-1}$ is identically equal to $1$ and the distance from $\partial \Omega_j$ to $\partial \Omega_{j-1}$ is larger than or equal to $ \varepsilon^{\delta_*}/100$.

Without loss of generality, we can assume that, for all $k \geq 1$
\[
\| \nabla^k \chi_j \|_{L^\infty (\mathbb R^{n+1})} \leq C \, \varepsilon^{-k \delta_*},
\]
for some constant $C_k >0$ only depending on $k$. 
 
\subsection{\bf A one parameter family of approximate solutions} Building on the analysis we have done in section~\ref{se:7}, the approximate solution $\tilde u_\varepsilon$ is defined by
\[
\tilde u_\varepsilon  : =  \chi_1  \, \bar u_\varepsilon \pm  (1 - \chi_1) ,
\]
where $\pm$ corresponds to whether the point belongs to $\mathbb R^{n+1}_\pm$.  Here the function $\bar u_\varepsilon$ is the one defined in (\ref{eqves}), namely
\[
Z^* \bar u_\varepsilon (y,t) : =  u_\varepsilon (t) .
\]
Observe that $\bar u_\varepsilon$ is exponentially close to $\pm 1$ at infinity and hence, it is reasonable to graft it to the constant functions $\pm 1$ away from some neighborhood of $\Gamma$. 

\subsection{\bf An infinite dimensional family of diffeomorphisms} Given a function $\zeta \in \mathcal C^{2, \alpha}( \Gamma)$, we define a  diffeomorphism $D_\zeta$ of $\mathbb R^{n+1}$ as follows. 
 \[
Z^* D_\zeta (y,t) = Z (y ,t -  \chi_2 (y, t) \, \zeta (y)),
\]
in $\Omega_2$ and 
\[
D_\zeta = {\rm Id} ,
\]
in $\mathbb R^{n+1} \setminus \Omega_2$. It is easy to check that this is a diffeomorphism of $\mathbb R^{n+1}$ provided the norm of $\zeta$ is small. 

Also, observe that, in $\Omega_2$, the inverse of $D_\zeta$ can be written as 
\[
Z^* D_\zeta^{-1} (y,t) = Z (y , t  +  \chi_2 (y, t) \, \zeta (y) + \eta (y,t, \zeta (y)) \, \zeta(y)^2 ), 
\]
where $(y,t, z ) \longmapsto \eta (y,t,z)$ is a smooth function defined for $z$ small (this follows at once from the inverse function theorem applied to the function 
\[
t  \longmapsto  t -  \chi_2 (y, t)  \, z ,
\]
details are left to the reader).

\subsection{\bf Rewriting the equation}  
Given a function $\zeta \in \mathcal C^{2, \alpha} (\Gamma)$, small enough, we use the diffeomorphism $D_\zeta$, we write $u = \bar u \circ D_\zeta$ so that the equation  
\[
\varepsilon^2 \, \Delta u + u -u^3 =0 ,
\]
can be rewritten as 
\begin{equation}
\varepsilon^2 \, (\Delta \bar u \circ D_\zeta ) \circ D_\zeta^{-1} + \bar u - \bar u^3 = 0 .
\label{neweq}
\end{equation}
Observe that, when $\chi_2 \equiv 1$, the diffeomorphism $D_\zeta$ is just given by  $Z^* D_\zeta (y,t) = Z(y, t-\zeta(y))$ and, as a consequence, in the coordinates $(y,t)$ this equation is precisely the one given in (\ref{eq:rfe}). Also observe that this equation is nonlinear in $\zeta$ and its partial derivatives, this is clear from (\ref{eq:rfe}). But, and this is a key point, since we have composed the whole equation with $D^{-1}_\zeta$, the function $\zeta$ does not appear anymore composed with the function $\bar u$. This property is due to the special structure of the diffeomorphism $D_\zeta$ and hence to the special structure of $D_\zeta^{-1}$. 

Now, we look for a solution of (\ref{neweq}) as a perturbation of  $\tilde u_\varepsilon$, and hence, we define  
\[
\bar u : = \tilde u_\varepsilon + v ,
\]
so that the equation we need to solve can now be written as 
\begin{equation}
\varepsilon^2 \, (\Delta v \circ D_\zeta ) \circ D_\zeta^{-1} + v - 3 \, \tilde u_\varepsilon^2 \, v +  E_\varepsilon  (\zeta) + Q_\varepsilon (v)  =0 ,
\label{lasteq}
\end{equation}
where 
\[
E_\varepsilon (\zeta) : = \varepsilon^2  (\Delta \tilde u_\varepsilon \circ D_\zeta ) \circ D_\zeta^{-1} + \tilde u_\varepsilon - \tilde u_\varepsilon^3,
\]
is the error corresponding to the fact that $\bar u_\varepsilon$ is an approximate solution and 
\[
Q_\varepsilon (v) : = v^3 + 3 \, \tilde u_\varepsilon \, v^2 ,
\]  
collects the nonlinear terms in $v$. Again, when $\chi_2 \equiv 1$, ${\tilde u}_\varepsilon = \bar u_\varepsilon$ and, in the coordinates $(y,t)$, the equation (\ref{lasteq}) is nothing but $\mathfrak N (v , \zeta )=0$ where the nonlinear operator $\mathfrak N$ has been defined in (\ref{eq:nrfe}).

Finally, in order to solve (\ref{lasteq}), we use a very nice trick which was already used in \cite{dkwdg}. This trick amounts to decompose the function $v$ into two functions, one of which $  \chi_4 \, v^\sharp$ is supported in a tubular neighborhood on $\Gamma$ and the other one $v^\flat$ being globally defined in $\mathbb R^{n+1}$, instead of solving (\ref{lasteq}), one solves a coupled system of equation. One of the equation involves the operator $L_\varepsilon$ acting on $v^\sharp$ and the operator $J_\Gamma$ acting on $\zeta$ while the other equation involves the operator $(\varepsilon^2 \, \Delta -2)$ actin of $v^\flat$. At first glance this might look rather counterintuitive but, as we will see, this strategy allows one to use directly the linear results we have proven in the previous sections. 

Therefore, we decompose
\[
v : =  \chi_4 \, v^\sharp + v^\flat ,
\]
where the function $v^\flat$ solves 
\[
\begin{array}{rllll}
\mathcal L_\varepsilon \, v^\flat & = &- (1- \chi_4) \,  \Big [ \varepsilon^2 \, \left( \Delta (v^\flat \circ D_\zeta) \circ D_\zeta^{-1}  - \Delta v^\flat \right) \\[3mm] 
& + &  3 \, (\tilde u_\varepsilon^2 -1) \, v^\flat  - E_\varepsilon (\zeta) - Q_\varepsilon (\chi_4 \, v^\sharp + v^\flat) \Big] \\[3mm]
& - & \varepsilon^2 \, \left( (\Delta ( (\chi_4 \, v^\sharp) \circ D_\zeta ) - \chi_4 \Delta  (v^\sharp \circ D_\zeta)  \right ) \circ D_\zeta^{-1} , 
\end{array}
\]
where $\mathcal L_\varepsilon$ has been defined in (\ref{eq:seo}).

For short, the right hand side will be denoted by $N_\varepsilon ( v^\flat, v^\sharp, \zeta  )$ so that this equation reads
\begin{equation}
\mathcal L_\varepsilon \, v^\flat = N_\varepsilon ( v^\flat, v^\sharp, \zeta  ) .
\label{eq:1}
\end{equation}
Observe that the right hand side of this equation vanishes in $\Omega_4$.

\begin{remark}
We know from Proposition~\ref{pr:10.91} that if 
\[
(\varepsilon^2 \, \Delta - 2 ) \, w =  f ,
\]
then 
\begin{equation}
\| w \|_{\mathcal C^{2, \alpha}_{\varepsilon, \nu} (\mathbb R^{n+1})} \leq C \, \| f\|_{\mathcal C^{0, \alpha}_{\varepsilon, \nu} (\mathbb R^{n+1})} .
\label{eq:espa}
\end{equation}
In the case where $f \equiv 0$ in $\Omega_4$, we can be more precise and we can show that the estimate for $w$ can be improved in $\Omega_5$. Indeed, we claim that we have
\[
\| \chi_5 \, w \|_{{\mathcal C^{2, \alpha}_{\varepsilon, \nu} (\mathbb R^{n+1})}} \leq C \, \varepsilon^2 \,  \| f\|_{\mathcal C^{0, \alpha}_{\varepsilon, \nu} (\mathbb R^{n+1})} ,
\]
provided $\varepsilon$ is small enough (as we will see the $\varepsilon^2$ can be replaced by any power of $\varepsilon$). Starting from (\ref{eq:espa}), his estimate follows easily from the construction of suitable barrier functions for the $\varepsilon^2 \, \Delta -2$. Indeed,  given a point $x^0 \in \mathbb R^{n+1}$, we can use 
\[
x : = (x_1, \ldots, x_{n+1}) \longmapsto \sum_{i=1}^{n+1} \cosh \left( \sqrt 2 \, \frac{(x_i -x_i^0)}{\varepsilon}Ê\right) ,
\]  
as a barrier in $\Omega_4$, to estimate $w$ at any point $x^0 \in  \Omega_5$ in terms of the estimate of $w$ on the boundary of ball of radius $\varepsilon^{\delta_*}/100$ centered at $x^0$. Performing this analysis at any point of $\Omega_5$, we conclude that 
\[
\| (1+ |x|^2)^{-\nu/2} \, w \|_{L^\infty (\Omega_5)} \leq ÊC \, \, e^{-  c^*Ê\,Ê\varepsilon^{\delta_*-1} } \, \| (1 + |x|^2)^{-\nu/2} \, w \|_{L^\infty (\Omega_4)} ,
\]  
where $c^*  : = \, \sqrt 2/100$. As usual, once the estimate for the $L^\infty$ norm has been derived, the estimates for the derivatives follow at once from Schauder's estimates. 

We can summarize this discussion by saying that, if $f\equiv 0$ in $\Omega_4$, then (\ref{eq:espa}) can be improved into 
\begin{equation}
\| w \|_{\tilde{\mathcal C}^{2, \alpha}_{\varepsilon, \nu} (\mathbb R^{n+1})} \leq C \, \| f\|_{\mathcal C^{0, \alpha}_{\varepsilon, \nu} (\mathbb R^{n+1})} ,
\label{eq:newest}
\end{equation}
where, by definition
\[
\|Êv  \|_{\tilde{\mathcal C}_{\varepsilon, \nu}^{2, \alpha} (\mathbb R^{n+1})}  : = \varepsilon^{-2} \, \|Ê\chi_5 \, v  \|_{\mathcal C_{\varepsilon, \nu}^{2, \alpha} (\mathbb R^{n+1})}   + \|Êv \|_{\mathcal C_{\varepsilon, \nu}^{2, \alpha} (\mathbb R^{n+1})} .
\]
\label{re:11.1}
\end{remark}

Taking the difference between the equation satisfied by $v$ and the equation satisfied by $v^\flat$, we find that it is enough that  $v^\sharp$ solves, 
\[
\begin{array}{rllll}
\varepsilon^2 \, \Delta (v^\sharp  \circ D_\zeta) \circ D_\zeta^{-1}  + v^\sharp - 3 \, \tilde u_\varepsilon^2 \, v^\sharp & = &  - E_\varepsilon  (\zeta) - Q_\varepsilon (\chi_4 \, v^\sharp + v^\flat) +  3 \, ( \tilde u_\varepsilon^2 - 1) \, v^\flat \\[3mm]
& - &  \varepsilon^2 \, \left( \Delta (v^\flat \circ D_\zeta) \circ D_\zeta^{-1}  - \Delta v^\flat \right)  ,
\end{array}
\]
in the support of $\chi_4$.  Since we only need this equation to be satisfied on the support of $\chi_4$, we can as well solve the equation 
\begin{equation}
\begin{array}{rllll}
L_\varepsilon v^\sharp - \varepsilon  \, J_\Gamma \, \zeta \, \dot u_\varepsilon & = & \displaystyle \chi_3 \, \Big[ L_\varepsilon \,  v^\sharp - \varepsilon^2 \,   \left( \Delta (v^\sharp  \circ D_\zeta) \circ D_\zeta^{-1} -  v^\sharp + 3 \, \bar u_\varepsilon^2 \, v^\sharp  \right)  \\[3mm]
& - & \varepsilon^2 \, \left( \Delta (v^\flat \circ D_\zeta) \circ D_\zeta^{-1}  - \Delta v^\flat \right)   \\[3mm]
& - & E_\varepsilon  (\zeta) - \varepsilon  \, J_\Gamma \, \zeta \, \dot u_\varepsilon - Q_\varepsilon (\chi_4 \, v^\sharp + v^\flat) +  3 \, ( \bar u_\varepsilon^2 - 1) \, v^\flat  \Big],
\end{array}
\label{eq:3}
\end{equation}
where the operator $L_\varepsilon$ is the one defined in (\ref{eq:12-1}). Here we have implicitly used the fact that $\tilde u_\varepsilon = \chi_3 \, \bar u_\varepsilon $ on the support of $\chi_3$.  For short, the right hand side will be denoted by $M_\varepsilon ( v^\flat, v^\sharp, \zeta  )$ so that this equation reads
\begin{equation}
\varepsilon^2 \, L_\varepsilon v^\sharp - \varepsilon  \, J_\Gamma \, \zeta \, \dot u_\varepsilon   = M_\varepsilon ( v^\flat, v^\sharp, \zeta  ) .
\label{eq:1}
\end{equation}

This last equation is now  projected over the space of functions satisfying (\ref{eq:12-10}) and the space of functions of the form $\dot u_\varepsilon$ times a function defined on $\Gamma$. Recall that we have defined $\Pi$ to be the orthogonal projection on $\dot u_\varepsilon$, namely
\[
\Pi (f)  : =  \frac{1}{\varepsilon \, c} \, \int_{\mathbb R}  f(y,t) \, \dot u_\varepsilon  (t) \, dt ,
\]
where the constant $c >0$ is explicitly given by
\[
c : =\frac{1}{\varepsilon} \,  \int_{\mathbb R}Ê\dot u_\varepsilon^2 (t) \, dt =  \int_{\mathbb R}Ê(u_1')^2 (t) \, dt ,
\]
and by $\Pi^\perp$ the orthogonal projection on the orthogonal of $\dot u_\varepsilon$, namely
\[ 
\Pi^\perp (f) (y,t) : =  f - \Pi (f) \, \dot u_\varepsilon (t) .
\]  
 
If we further assume that $v^\sharp$ satisfies (\ref{eq:12-10}), then (\ref{eq:3}) is equivalent to the system
\begin{equation}
\begin{array}{rllll}
L_\varepsilon v^\sharp & = & \Pi^\perp \, \Big[M_\varepsilon ( v^\flat, v^\sharp, \zeta  )  \Big],
\end{array}
\label{eq:3}
\end{equation}
and
\begin{equation}
\begin{array}{rllll}
- \varepsilon  \, J_\Gamma \, \zeta & = & \Pi \, \Big[M_\varepsilon ( v^\flat, v^\sharp, \zeta  )  \Big].
\end{array}
\label{eq:3}
\end{equation}

\subsection{The existence of a solution}
\label{se:e10}
We summarize the above discussion. We are looking for a solution of 
\begin{equation}
\varepsilon^2 \, \Delta u + u-u^3 =0 ,
\label{eq:qs}
\end{equation}
of the form
\[
u = \left( \tilde u_\varepsilon + \chi_4 \, v^\sharp + v^\flat \right) \circ D_\zeta ,
\]
where the function $v^\sharp$ is defined on $\Gamma \times \mathbb R$, the function $v^\flat$ is defined in $\mathbb R^{n+1}$ and the function $\zeta$ is defined on $\Gamma$. In this case, (\ref{eq:qs}) is equivalent to the solvability of the coupled system
\begin{equation}
\left\{
\begin{array}{rllll}
\mathcal L_\varepsilon\, v^\flat  & = & N_\varepsilon ( v^\flat, v^\sharp, \zeta ) \\[3mm]
L_\varepsilon v^\sharp & = & \Pi^\perp \, \Big[M_\varepsilon ( v^\flat, v^\sharp, \zeta  )  \Big] \\[3mm]
- \varepsilon  \, J_\Gamma \, \zeta & = & \Pi \, \Big[M_\varepsilon ( v^\flat, v^\sharp, \zeta  )  \Big].
\end{array}
\right.
\label{eq:1111}
\end{equation}

Closer inspection of the construction of the approximate solution shows that~: 
\begin{lemma}
The following estimates hold
\[
\|ÊN_\varepsilon  (0,0,0) \|_{\mathcal C_{\varepsilon, -2}^{0, \alpha} (\mathbb R^{n+1})} + \| \Pi^\perp \, ( M_\varepsilon (0,0,0) )  \|_{\mathcal C^{0, \alpha}_{\varepsilon, -2} (\Gamma \times \mathbb R)} \leq  C \, \varepsilon^2  .
\]
Moreover, given $\nu \geq -1$, we have
\[
\| \Pi \, ( M_\varepsilon (0,0,0) )  \|_{\mathcal C^{0, \alpha}_{\nu-2} (\Gamma)} \leq  C \, \varepsilon^3 .
\]
\label{le:leend}
\end{lemma}
\begin{proof}
Since $v^\sharp =0$, $v^\flat =0$ and $\zeta =0$, the estimate follow from the understanding of 
\[
E_\varepsilon (0) = \varepsilon^2  \, \Delta \tilde u_\varepsilon + \tilde u_\varepsilon - \tilde u_\varepsilon^3 .
\]
But, in the range where $\chi_1 \equiv 1$, we have already seen that 
\[
\varepsilon^2  \, \Delta  u_\varepsilon +  u_\varepsilon - u_\varepsilon^3  =  -  \varepsilon \, H_t  \,   \dot u_\varepsilon ,
\]
and the estimates then follow at once from (\ref{eq:7.21}) in Lemma~\ref{le:7.1} and (\ref{eq:7.22}) in Lemma~\ref{le:7.2}.
\end{proof}

We also need the 
\begin{lemma}
Assume that $\nu \in  [-1, 0)$. Then, there exists $\delta >0$ (independent of $\alpha \in (0,1)$) such that the following estimates hold
\[
\begin{array}{lll}
\|  N _\varepsilon (v^\flat_2, v^\sharp_2, \zeta_2 )  -  N_\varepsilon (v^\flat_1 , v^\sharp_1 , \zeta_1 ) \|_{\mathcal C^{0, \alpha}_{\varepsilon, -2} (\mathbb R^{n+1})}   \\[3mm]
\qquad \leq C \, \varepsilon^{\delta} \, \left( \| v^\flat_2 - v^\flat_1\|_{{\mathcal C}^{2, \alpha}_{\varepsilon, -2} (\mathbb R^{n+1})} + \| v^\sharp_2 - v^\sharp_1\|_{\mathcal C^{2, \alpha}_{\varepsilon, -2} (\Gamma \times \mathbb R)} +  \| \zeta_2 - \zeta_1\|_{\mathcal C^{2, \alpha}_{\nu} (\Gamma)} \right) 
\end{array}
\]
\[
\begin{array}{lll}
\| \Pi^\perp \, ( M_\varepsilon (v^\flat_2, v^\sharp_2, \zeta_2 ) -  M_\varepsilon (v^\flat_1, v^\sharp_1, \zeta_1 ))  \|_{\mathcal C^{0, \alpha}_{\varepsilon, -2} (\Gamma \times \mathbb R)}  \\[3mm]
\qquad \leq C \, \varepsilon^{\delta} \, \left( \| v^\flat_2 - v^\flat_1\|_{\tilde{\mathcal C}^{2, \alpha}_{\varepsilon, -2} (\mathbb R^{n+1})} + \| v^\sharp_2 - v^\sharp_1\|_{\mathcal C^{2, \alpha}_{\varepsilon, -2} (\Gamma \times \mathbb R)} +  \| \zeta_2 - \zeta_1\|_{\mathcal C^{2, \alpha}_{\nu} (\Gamma)} \right) 
\end{array}
\]
and
\[
\begin{array}{lll}
\|Ê\Pi \, ( M_\varepsilon (v^\flat_2, v^\sharp_2, \zeta_2 ) - M_\varepsilon (v^\flat_1, v^\sharp_1, \zeta_1 ) ) \|_{\mathcal C^{0, \alpha}_{\nu-2} (\Gamma)}  \\[3mm]
\qquad \leq C  \, \varepsilon^{1-\alpha} \, \| v^\sharp_2 - v^\sharp_1\|_{\mathcal C^{2, \alpha}_{\varepsilon, -2} (\Gamma \times \mathbb R)} + C \, \varepsilon^{1+\delta} \, \left( \| v^\flat_2 - v^\flat_1\|_{\tilde{\mathcal C}^{2, \alpha}_{\varepsilon, -2} (\mathbb R^{n+1})}  +  \| \zeta_2 - \zeta_1\|_{\mathcal C^{2, \alpha}_{\nu} (\Gamma)} \right) .
\end{array}
\]
\end{lemma}
\begin{proof}
The proof is rather technical but does not offer any real difficulty. Observe that, in the last two estimates,  the use of the norm  $\|Êv^\flat \|_{\tilde{\mathcal C}_{\varepsilon, -2}^{2, \alpha} (\mathbb R^{n+1})} $ instead of $\|Êv^\flat \|_{{\mathcal C}_{\varepsilon, -2}^{2, \alpha} (\mathbb R^{n+1})} $  is crucial to estimate the term $-  3 \, ( \bar u_\varepsilon^2 - 1) \, v^\flat $ in the definition of $M_\varepsilon (v^\flat, v^\sharp, \zeta)$. In the last estimate, the first term on the right hand side comes from the estimate of the projection of $\varepsilon^2 \, (\Delta_{g_{t}} - \Delta_{\mathring g} ) \, v^\sharp$ which induces a loss of $\varepsilon^{\alpha}$.
\end{proof}

We choose, $\nu \in [-1, 0)$, $\nu > \nu_0^+$ and further assume that $\nu$ is not an indicial root of $J_\Gamma$. We choose $\alpha \in (0,1)$ small enough, namely such that $2 \alpha < \delta$, the constant which appears in the last Lemma). We use the result of Proposition~\ref{pr:8.3}, Proposition~\ref{pr:10.91} (as well as (\ref{eq:newest})  in Remark~\ref{re:11.1}) and Proposition~\ref{pr:10.10} and Lemma~\ref{le:1010}, to rephrase the solvability of (\ref{eq:1111}) as a fixed point problem. 

Theorem~\ref{main2} is now a simple consequence of the application of a fixed point theorem for contraction mapping which leads to the existence of a unique solution 
\[
{\bf u}_\varepsilon =  (\bar u_\varepsilon + \chi_4 \, v^\sharp + v^\flat ) \circ D_\zeta ,
\]
where 
\[
\|Êv^\flat \|_{\tilde{\mathcal C}_{\varepsilon, -2}^{2, \alpha} (\mathbb R^{n+1})} + \| v^\sharp \|_{\mathcal C^{2, \alpha}_{\varepsilon, -2} (\Gamma \times \mathbb R)} + \varepsilon^{2\alpha} \,  \| \zeta  \|_{\mathcal C^{2, \alpha}_{\nu} (\Gamma)} \leq  C \, \varepsilon^2  .
\]
We leave the details for the reader.  

As a byproduct, we have a rather good control on the zero set of the solution. Indeed, following the different steps of the proof, one can see that, at infinity, the zero set of ${\bf u}_\varepsilon$ is a normal over $\Gamma$ for a function which is bounded by a constant times $\varepsilon^{2-\alpha} \, (d_\Gamma)^{\nu}$. Since we do not need this result, we shall leave its proof to the interested reader.

\begin{remark}
As already mentioned in the introduction, our construction extends to the case where the minimal hypersurface $\Gamma$ is asymptotic to a minimal cone which is not necessarily minimizing. In this broader context, in order to apply our  construction, one needs to be able to apply the result of Proposition~\ref{pr:10.10} and get surjectivity of $J_\Gamma$ in the space $\mathcal C^{2, \alpha}_\nu (\Gamma)$ for some $\nu < 0$. However, this surjectivity boils down to  ask injectivity of $J_\Gamma$ in the space $\mathcal C^{2, \alpha}_{\nu'} (\Gamma)$ for some $\nu' > 2-n$. Under this latter assumption, our construction applies.  
\end{remark}

\section{The proof of Theorem~\ref{main1}}

We give now the proof of Theorem~\ref{main1}. The proof builds on the previous construction which lead to Theorem~\ref{main2} using a special minimal surface. 

We assume that $m \geq 4$ and we follow the analysis of section~\ref{se:3}. In particular, we consider the strictly area minimizing cone $C_{m,m}$ which was defined in Example~\ref{ex:1}. In this special case, using the notations of (\ref{metriccone}) and (\ref{metricconebis}), we see that the induced metric on the minimal cone $C_{m,m}$ is given by 
\[
g =  e^{2t} \, \left( dt^2 + \frac 1 2 \, (g_1 + g_2) \right) ,
\]
where $g_1$ and $g_2$ are the standard metrics on $S^m$, while the second fundamental form on ${m,m}$ is given by
\[
h =  \frac 1 2 \, e^{t} \,  (g_1 - g_2).
\]
In particular, this implies that 
\[
{\rm Tr}_g h^{(2k+1)} = 0 ,
\]
for all $k \in \mathbb N$. In other words, because of the symmetries of $C_{m,m}$, the principal curvatures of $C_{m,m}$ are either $0$, with multiplicity $1$), $e^{-t}$ with multiplicity $m$ and $-e^{-t}$ again with multiplicity $m$ and hence, the trace of $h^{(j)} $ is zero if $j$ is odd.

Now, we make use of  the result of Theorem~\ref{th:4.2} which guaranties the existence of a minimal hypersurface $\Gamma_m$ which sits on one side of $C_{m,m}$ and which is a normal graph over $C_{m,m}$ for a function which decays like $e^{\nu_0^+ t}$. As a consequence, we see that the principal curvatures of $\Gamma_m$ are as follows : there is one principal curvature which can be estimated by $\mathcal O (e^{(\nu_0^+-2)t})$ and $m$ principal curvatures whcih can be estimated by $\pm e^{-t} + \mathcal O (e^{(\nu_0^+-2)t})$. In particular, this implies that the result of Lemma~\ref{le:7.2} can be improved into~:
\begin{lemma}
For all $k  \geq 0$, there exists a constant $C_{k } >0$ such that 
\begin{equation}
	 |  \nabla^{k}  \, \Pi  \, ( \chi_3 \, \mathfrak N (0 ,0 )) |_{ \mathring g} \leq C_{k} \, \varepsilon^{2} \,  (d_{\Gamma_m})^{\nu_0^+ -4 - k} ,
\label{eq:7.222} 
\end{equation}
in $\Gamma_m$.
\end{lemma}
\begin{proof}
This uses the fact that $\Gamma_m$ is a normal graph over $C_{m,m}$ for a function which decays like $(d_{\Gamma_m})^{\nu_0^+}$. 
\end{proof}

One should compare this estimate with the estimate (\ref{eq:7.22}) which holds in a more general setting and which was used in the proof of the third estimate in Lemma~\ref{le:leend}. Observe that it was because of (\ref{eq:7.22}) that we needed to use $\nu \geq -1$ in the weighted spaces where we were looking for a solution of (\ref{eq:1111}).  

Recall that in the case where the cone $C_{m,m}$ is minimizing, we have
\[
-2  \leq \nu_0^+ \leq -1 ,
\]
and hence $\nu_0^+-2 < -3$.  As a consequence, in the case where the minimal hyperfurface we start with is given by $\Gamma_{m}$ we can use some parameter $\nu$ such that 
\[
- 2 < \nu < \nu_0^+ ,
\] 
in the proof of Theorem~\ref{main2}. Recall that, in the case where the minimal cone is strictly area minimizing, the operator $J_\Gamma$ is injective in $\mathcal C^{2,Ê\alpha}_\nu(\Gamma)$ for $\nu < \nu_0^+$ and hence it is surjective for all $\nu > \nu_0^-$.

We apply the construction described above to the hypersurfaces $(1+\lambda) \, \Gamma_m$ for any $\lambda$ close enough to $0$ and denote by  ${\bf u}_{\varepsilon, \lambda}$ the solution which is known to exist for all $\varepsilon$ small enough. It should be clear the ${\bf u}_{\varepsilon, \lambda}$ depends smoothly on $\lambda$, at least when $\varepsilon$ is fixed close enough to $0$. Differentiation with respect ot $\lambda$, at $\lambda =0$, yields 
\[
(\varepsilon^2 \, \Delta  +1 -  3 \, {\bf u}_{\varepsilon, \lambda}^2 ) \, \phi_\varepsilon =0 ,
\]
where 
\[
\phi_\varepsilon : = \partial_\lambda  {\bf u}_{\varepsilon, \lambda } \, _{| \lambda =0} .
\]

We claim that~:
\begin{proposition}
For $\varepsilon$ small enough, the function $\phi_\varepsilon$ is positive.
\end{proposition}
\begin{proof}
We fix $c^* >0$ large enough and we define $T_\varepsilon(\Gamma_m)$ to be the tubular neighborhood around $\Gamma$ of width $c^* \, \varepsilon$. We first prove that $\phi_\varepsilon >0$ in $T_\varepsilon (\Gamma_m)$.

Recall that, in the proof of Lemma~\ref{le:1010},  we have defined 
\[
\zeta_0 : =  y \cdot N(y),
\]
which is positive and is bounded from above and from below  by a positive constant times $(d_{\Gamma_m} )^{\nu_0^+}$.

Now, we use the fact that, at infinity, the rate of convergence of $\Gamma_m$ to $C_m$ is controlled by a constant times $(d_\Gamma)^{\nu_0^+}$ and hence $(1+\lambda) \, \Gamma_m$ is itself a normal graph over $\Gamma_m$ for some function which, for $\lambda$ small enough, is bounded by a constant times $ \lambda \, \zeta_0$. Using this property and following the different steps of the construction of  ${\bf u}_{\varepsilon, \lambda}$, one can check that
\[
\partial_\lambda {\bf u}_{\varepsilon, \lambda} \, _{| \lambda =0}  =  \varepsilon^{-1} \, \dot u_\varepsilon  \, \left( \zeta_0   + \mathcal O (\varepsilon^{2-2 \alpha} \, (d_{\Gamma_m})^{\nu})  \right) + \mathcal O (\varepsilon^2 \, (d_{\Gamma_m})^{\nu}) , 
\]
in $T_\varepsilon (\Gamma_m)$. Hence, $\phi_\varepsilon >0$ in $T (\Gamma_m)$ provided $\varepsilon$ is chosen small enough. This is where we use in a fundamental way the fact that, in the proof of Theorem~\ref{main2}, we can now use some weight parameter $\nu < \nu_0^+$. 

Then, provided we have fixed $c^* >0$ large enough, the maximum principle can be used to prove that $\phi_\varepsilon >0$ in $\mathbb R^{n+1}$ and this complete the proof of the result.
\end{proof}

As explain in the introduction, this implies that the solution ${\bf u}_{\epsilon, 0}$ is stable, for all $\varepsilon$ small enough and this completes the proof of Theorem~\ref{main1}.


\begin{thebibliography}{999999}
 
 \bibitem{aac} G. Alberti, L. Ambrosio and X. Cabr\'e, {\em On a long-standing conjecture of E. De Giorgi: symmetry in 3D for general nonlinearities and a local minimality property}, Acta Appl. Math. 65 (2001) (Special issue dedicated to Antonio Avantaggiati on the occasion of his 70th birthday), no. 1-3, 9--33.

\bibitem{Ale} H. Alencar, {\em Minimal hypersurfaces of $\mathbb R^{2m}$ invariant by $SO(m) \times SO (m)$}, Trans. AMS. 137 (1993), 129-141.

\bibitem{Ale-Bar}  H. Alencar, A. Barros, O. Palmas, J. G. Reyes and W. Santos, {\em $O(m)\times O(n)$-invariant minimal hypersurfaces in $R^{m+n}$}, Annals of Global Analysis and Geometry 27 (2005), 179-199.

\bibitem{All-Cah} S. M. Allen and J. W. Cahn. {\em A microscopic theory for antiphase boundary motion and its application to antiphase domain coarsening}. Acta Metall. Mater., 27, (1979), 1085-1095.

\bibitem{cabre} L. Ambrosio and X. Cabr\'e, {\em Entire solutions of semilinear elliptic equations in $\mathbb R^3$ and a conjecture of De Giorgi,} Journal Amer. Math. Soc. 13 (2000), 725--739.

\bibitem{BBG} M. T. Barlow, R. F. Bass and C. Gui, {\em The Liouville property and a conjecture of De Giorgi}, Comm. Pure Appl.  Math. 53(8)(2000), 1007-1038.

\bibitem{BHM} H. Berestycki, F. Hamel and R. Monneau, {\em One-dimensional symmetry of bounded entire solutions of some elliptic equations}, Duke Math. J. 103 (2000), 375-396.

\bibitem{bdg} E. Bombieri, E. De Giorgi and E. Giusti, {\em Minimal cones and the Bernstein problem,} Invent. Math. 7 1969 243--268.

\bibitem{Cab} X. Cabr\'e, {\em Uniqueness and stability of saddle-shaped solutions to the Allen-Cahn equation}, preprint (2011). arXiv:1102.3111
 
\bibitem{Cab-Ter-1} X. Cabr\'e and  J. Terra, {\em Saddle-shaped solutions of bistable diffusion equations in all of $R^{2m}$. } Jour. of the European Math. Society 11, no. 4,  (2009), 819-843.

\bibitem{Cab-Ter-2} X. Cabr\'e and  J. Terra, {\em Qualitative properties of saddle-shaped solutions to bistable diffusion equations}, arxiv.org/abs/0907.3008, to appear in Communications in Partial Differential Equations.

\bibitem{caffarellicordoba} L. Caffarelli and A. C\'ordoba, {\em Uniform convergence of a singular perturbation problem,} Comm. Pure Appl. Math. XLVII (1995), 1-12.

\bibitem{Caf-Har-Sim} L. Caffarelli, R. Hardt and L. Simon, {\em Minimal surfaces with isolated singularities}. Manuscripta Math. 48, no. 1-3,  (1984), 1-18.

\bibitem{Cha} C. Chan, {\em Complete minimal hypersurfaces with prescribed asymptotics at infinity}, J. Reine Angew. Math. 483, (1997), 163-181.

\bibitem{dancer} N. Dancer, private communication. 

\bibitem{Dan-Fif-Pel} H. Dang, P.C. Fife and L.A. Peletier, {\em Saddle solutions of the bistable di?usion equation},  Z. Angew. Math. Phys. 43, no. 6, (1992), 984-998. 

\bibitem{Dav} D. Davini, {\em Calibrations of Lawson's Cones},  Rend. Sem.  Mat. Univ. Padova 111 (2004), 55-70.

\bibitem{dkwdg} {M. del Pino, M. Kowalczyk and J. Wei,} {\em A conjecture by  de Giorgi in large dimensions.} Preprint 2008.

\bibitem{dkw-2} {M. del Pino, M. Kowalczyk and J. Wei} {\em Entire solutions of the Allen-Cahn equation and complete embedded minimal surfaces of finite total curvature}ÊPreprint (2008).

\bibitem{dkpw2} M. del Pino, M. Kowalczyk,  F. Pacard and J. Wei, {\em Multiple-end  solutions to the Allen-Cahn equation in $\mathbb R^2$}, to appear in Jour. Func. Anal.
 
\bibitem{dg}
E. de Giorgi, {\em Convergence problems for functionals and operators,} Proc. Int. Meeting on Recent Methods in Nonlinear Analysis (Rome, 1978), 131-188, Pitagora, Bologna (1979).

\bibitem{Fa}  A. Farina, {\em Symmetry for solutions of semilinear elliptic equations in $\mathbb R^N$ and related conjectures}, Ricerche Mat. 48 (suppl.) (1999), 129-154.

\bibitem{FV} A. Farina and E. Valdinoci, {\em The state of art for a conjecture of De Giorgi and related questions}, to appear in ``Reaction-Diffusion Systems and  Viscosity Solutions'', World Scientific, 2008.

\bibitem{Fic-Col-Sch} D. Fischer-Colbrie and R. Schoen, {\em The structure of complete stable minimal surfaces in $3$-manifolds of nonnegative scalar curvature}, Comm. Pure and Applied Maths. 33, (1980), 199-211.

\bibitem{Har-Sim} R. Hardt and L. Simon, {\em Area minimizing hypersurfaces  with isolated singularities}, J. Reine Angew. Math. 362 (1985), 102-129. 

\bibitem{gg} N. Ghoussoub and C. Gui, {\em On a conjecture of De Giorgi and some related problems,} Math. Ann. 311 (1998), 481-491.

\bibitem{gg2} N. Ghoussoub and C. Gui, {\em On De Giorgi's conjecture in dimensions $4$ and $5$.} Ann. of Math. (2) 157 (2003), no. 1, 313-334.

\bibitem{jerison} D. Jerison and R. Monneau, {\em Towards a counter-example to a conjecture of De Giorgi in high dimensions,} Ann. Mat. Pura Appl. 183 (2004), 439-467.

\bibitem{Kap-Seo} N. Kapouleas and S.D. Yang, {\em Minimal surfaces in the three-sphere by doubling the Clifford torus}. Amer. J. Math. 132,  no. 2, (2010), 257-295. 

\bibitem{Law} B. Lawson, {\em The equivariant Plateau problem and interior  regularity,} Trans. AMS, 173 (1972)

\bibitem{Lawl} G. Lawlor, {\em A sufficient condition for a cone to be area  minimizing}, Mem. AMS, (446),  91, (1991)

\bibitem{Nik-Uva} A. F. Nikiforov and V. B. Uvarov. {\em Special functions of mathematical physics}.  Birkhauser Verlag, Basel, 1988. 

\bibitem{Pac} F. Pacard,  {\em Lectures on "Connected sum constructions in geometry and nonlinear analysis"}, Preprint.

\bibitem{Pac-Rit} F. Pacard et M. Ritor\'e, {\em From constant mean curvature hypersurfaces to the gradient theory of phase transitions}, Journal of Differential Geometry, 64, (2003), 359-423.

\bibitem{Pit-Rub} J. T. Pitts and J. H. Rubinstein, {\em Equivariant minimax and minimal surfaces in geometric three-manifolds}, Bull. Amer. Math. Soc. (N.S.) 19, no. 1, (1988), 303-309.

\bibitem{savin} O. Savin, {\em Regularity of flat level sets in phase transitions}.  To appear in  Ann. of Math.

\bibitem{Sim-Sol} L. Simon and B. Solomon, {\em Minimal hypersurfaces asymptotic to quadratic cones in $\mathbb R^{n+1}$}, Invent. Math. 86 (1986), 535-551.

\bibitem{simons} J. Simons, {\em Minimal varieties in riemannian manifolds.} Ann. of Math. (2) 88 (1968) 62--105.

\bibitem{Sol} B. Solomon, {\em On foliations of $\mathbb R^{n+1}$ by minimal hypersurfaces}. Comm. Math. Helv. 61 (1986), 67-83.

\end{thebibliography}
\end{document}